\documentclass{./Article-eng}

\usepackage[all]{xy}
\usepackage{hyperref}
\hypersetup{colorlinks,allcolors=blue}
\usepackage{tikz-cd}

\makeatletter

\newcommand{\Rmnum}[1]{\expandafter\@slowromancap\romannumeral #1@}
\makeatother

\let\Gamma\varGamma
\let\Delta\varDelta

\begin{document}

\title[Strictly nef]{Strictly nef divisors on singular threefolds}


\author{
Juanyong Wang
}
\address{
Juanyong Wang\\
HCMS, Academy of Mathematics and Systems Science, Chinese Academy of Sciences, Beijing, 100190, China.}
\email{juanyong.wang@amss.ac.cn}

\author{Guolei Zhong
}
\address{
Guolei Zhong\\
Department of Mathematics, National University of Singapore, 10 Lower Kent Ridge Road, Singapore 119076, Republic of Singapore.}
\email{zhongguolei@u.nus.edu}

\date{}

\maketitle
\paragraph{Abstract:} 
Let $X$ be a normal projective variety with only klt singularities, and $L_X$ a strictly nef $\mathbb{Q}$-divisor on $X$. 
In this paper, we study the singular version of Serrano's conjecture, i.e., the ampleness of $K_X+t L_X$ for sufficiently large $t\gg 1$.
We show that, if $X$ is assumed to be a $\mathbb{Q}$-factorial Gorenstein terminal threefold, then $K_X+tL_X$ is ample for  $t\gg 1$ unless $X$ is a weak Calabi-Yau variety  (i.e., the canonical divisor $K_X\sim_\mathbb{Q}0$ and the augmented irregularity $q^\circ(X)=0$) with $L_X\cdot c_2(X)=0$.

\tableofcontents

\section{Introduction}
\label{sec_intro}

In this paper, we work over the field  $\mathbb{C}$ of complex numbers.  
Recall that a $\QQ$-Cartier divisor $L$ over a projective variety $X$ is said to be \emph{strictly nef} (resp. \emph{nef}) if  $L\cdot C>0$ (resp. $L\cdot C\geqslant 0$) for every (complete) curve $C$ on $X$. 
It is clear that ample divisors are always strictly nef, but the inverse is in general not true; cf.~e.g. \cite[Example 10.6]{Har70}.  
In 1990s, Campana and Peternell conjectured that a smooth projective variety $X$ with strictly nef anti-canonical divisor $-K_X$ is ample (cf.~\cite[Problem 11.4]{CP91}), which leads people to learn how far a strictly nef divisor is from being ample. 
To the best knowledge of the authors, this conjecture has an affirmative answer when $\dim X\le 3$ (cf.~\cite{Mae93, Ser95, Ueh00}). 
We refer readers to a recent joint paper by Liu, Ou, Yang and the authors \cite{LOWYZ21} for the singular (pair) version of this conjecture; moreover, an affirmative answer to the threefold case has been given in \cite[Theorem D]{LOWYZ21} (cf.~\cite[Corollary 1.8]{HL20} for the surface case). 

To study the conjecture of Campana and Peternell, Serrano proposed a more general conjecture that a strictly nef divisor on a smooth projective variety is ample after a small turbulence of the canonical divisor (cf.~\cite[Question 0.1]{Ser95}).
Meanwhile, Serrano verified this conjecture for abelian varieties, Gorenstein surfaces and smooth threefolds with some possible exceptions (cf.~\cite[Proposition 1.4 and Theorems 2.3, 4.4]{Ser95}). 
One decade later, together with \cite{Ser95}, Campana, Chen and Peternell confirmed this  conjecture in dimension three with the only (possible) exception that $X$ is Calabi-Yau and $L_X\cdot c_2=0$ (cf.~\cite[Theorem 0.4]{CCP08}). 
We refer readers to \cite{LS20, LM21} for recent progress on the case of Calabi-Yau manifolds. 
In the spirit of Serrano's conjecture, we consider the (pair case of) singular varieties and recall the following question which was proposed in \cite[Question 1.4]{LOWYZ21}. 
This is the initial point of this paper. 
\begin{ques}[Singular Version of Serrano's Conjecture]\label{main-conj-singular-arbitrary}
Let $(X,\Delta)$ be a projective klt pair, and $L$ a strictly nef $\mathbb{Q}$-divisor on $X$.
Is $K_X+\Delta+tL$  ample for sufficiently large $t\gg 1$?
\end{ques}

Ample divisors being strictly nef, {\hyperref[main-conj-singular-arbitrary]{Question    \ref*{main-conj-singular-arbitrary}}} can also be regarded as a weak analogue of Fujita's conjecture.
As a warm up, the following proposition gives a positive answer to 
{\hyperref[main-conj-singular-arbitrary]{Question    \ref*{main-conj-singular-arbitrary}}} for the surface case  (cf.~\cite[Corollary 1.8]{HL20}):

\begin{prop}[{cf.~\cite[Corollary 1.8]{HL20}}]\label{main_thm_surface}
Let $(X,\Delta)$ be a projective klt surface pair.
Suppose that $L_X$ is a strictly nef Cartier divisor on $X$.
Then $K_X+\Delta+tL_X$ is ample for every real number $t>3$.
\end{prop}

Indeed, \cite[Corollary 1.8]{HL20} showed that {\hyperref[main_thm_surface]{Proposition  \ref*{main_thm_surface}}} is true even when the projective surface pair $(X,\Delta)$ is only assumed to be log canonical.   
We refer readers to \cite[Proposition 5.4]{LOWYZ21} (={\hyperref[prop_Q-Goren_surface]{Proposition \ref*{prop_Q-Goren_surface}}})   for a further extension to the non-normal case. 
In higher dimensions however, there are only some partial results on  {\hyperref[main-conj-singular-arbitrary]{Question \ref*{main-conj-singular-arbitrary}}}. 
Except for the case when $X$ is a smooth threefold with the boundary $\Delta=0$ (cf.~\cite[Theorem 0.4]{CCP08} and \cite[Theorem 4.4]{Ser95}), 
the joint paper \cite[Theorems A and B]{LP20A} also gives a partial answer to {\hyperref[main-conj-singular-arbitrary]{Question \ref*{main-conj-singular-arbitrary}}} when  $K_X+\Delta$ is pseudo-effective.

In this paper, we study {\hyperref[main-conj-singular-arbitrary]{Question \ref*{main-conj-singular-arbitrary}}} with the main result  {\hyperref[main_theorem_Goren_ter_3fold]{Theorem   \ref*{main_theorem_Goren_ter_3fold}}} below. 
This gives  an extension  of \cite[Theorem 0.4]{CCP08} to the singular (threefold) case in the way of running the minimal model program (MMP for short).

\begin{thm}\label{main_theorem_Goren_ter_3fold}
Let $X$ be a normal projective threefold with only klt singularities, 
and $L_X$ a strictly nef $\mathbb{Q}$-Cartier divisor on $X$.
Suppose that one of the following conditions holds.
\begin{enumerate}
\item[(1)] The Kodaira dimension $\kappa(X)\geqslant 1$ (cf.~{\hyperref[kappa>=1]{Theorem   \ref*{kappa>=1}}}). 
\item[(2)] The augmented irregularity $q^{\circ}(X)>0$, and either $X$ is $\mathbb{Q}$-factorial or $X$ has only canonical singularities (cf.~{\hyperref[coro_q0>0]{Corollary   \ref*{coro_q0>0}}}). 
\item[(3)] The Kodaira dimension $\kappa(X)=0$, $X$ has only terminal singularities, and $L_X\cdot c_2(X)\neq 0$ (cf.~{\hyperref[thm-singular-cy]{Theorem  \ref*{thm-singular-cy}}}). 
\item[(4)] $X$ is uniruled with only isolated $\mathbb{Q}$-factorial Gorenstein canonical singularities (cf.~{\hyperref[remark_conclude_excep_tocurve]{Remark   \ref*{remark_conclude_excep_tocurve}}} and {\hyperref[thm_3fold_surface_curve]{Theorem    \ref*{thm_3fold_surface_curve}}}).
\end{enumerate}
Then $K_X+tL_X$ is ample for sufficiently large $t\gg 1$.
\end{thm}

We recall here that, the \textit{augmented irregularity} $q^\circ(X)$ is defined as the maximum of the irregularities $q(\widetilde{X}):=\textup{h}^1(\widetilde{X},\mathcal{O}_{\widetilde{X}})$, where $\widetilde{X}\to X$ runs over the (finite) quasi-\'etale covers of $X$ (cf.~\cite[Definition 2.6]{NZ10}).

To prove {\hyperref[main_theorem_Goren_ter_3fold]{Theorem    \ref*{main_theorem_Goren_ter_3fold}}}, we follow the ideas and strategies in \cite{Ser95} and \cite{CCP08}, though in our situation, the appearance of singularities will make things more complicated. 
\begin{rmq}\label{rmk-pf-mainthm}
We give a few remarks on the proof of {\hyperref[main_theorem_Goren_ter_3fold]{Theorem    \ref*{main_theorem_Goren_ter_3fold}}}.
\begin{enumerate}
\item[(1)] {\hyperref[main_theorem_Goren_ter_3fold]{Theorem \ref*{main_theorem_Goren_ter_3fold}}} is obtained as the union of {\hyperref[kappa>=1]{Theorem   \ref*{kappa>=1}}},  
{\hyperref[coro_q0>0]{Corollary  \ref*{coro_q0>0}}},
{\hyperref[thm-singular-cy]{Theorem  \ref*{thm-singular-cy}}}, 
{\hyperref[remark_conclude_excep_tocurve]{Remark   \ref*{remark_conclude_excep_tocurve}}}  and {\hyperref[thm_3fold_surface_curve]{Theorem    \ref*{thm_3fold_surface_curve}}}.
\item[(2)] Considering {\hyperref[prop-klt-eff-ample]{Proposition  \ref*{prop-klt-eff-ample}}}, we can reduce the proof of {\hyperref[main_theorem_Goren_ter_3fold]{Theorem    \ref*{main_theorem_Goren_ter_3fold} (3)}} to the case when $K_X\equiv 0$; for otherwise, $K_X$ is numerically equivalent to a non-zero effective divisor, and hence $K_X+tL_X$ is ample for $t\gg 1$. 
Then it follows from the abundance that $K_X\sim_\mathbb{Q}0$, i.e., $mK_X\sim 0$ for some $m\in\mathbb{N}$ (cf.~\cite[Theorem 1.3]{Gon13} or \cite{Nak04}). 
\item[(3)] In {\hyperref[main_theorem_Goren_ter_3fold]{Theorem    \ref*{main_theorem_Goren_ter_3fold} (4)}}, the condition on singularities is  to avoid flips when we run the MMP; see  {\hyperref[lem_del14]{Lemma  \ref*{lem_del14}}}, {\hyperref[lem_canonical_terminal]{Lemma  \ref*{lem_canonical_terminal}}}, and {\hyperref[rem_composition_conic]{Remark  \ref*{rem_composition_conic}}}.  
\end{enumerate}
\end{rmq}

Terminal threefolds having isolated singularities (cf.~\cite[Corollary 5.18]{KM98}),  the following corollary follows from  {\hyperref[main_theorem_Goren_ter_3fold]{Theorem   \ref*{main_theorem_Goren_ter_3fold}}}  immediately  (cf.~{\hyperref[rmk-pf-mainthm]{Remark    \ref*{rmk-pf-mainthm} (2)}}).
Recall that, a  normal projective variety $X$ with only klt  singularities is said to be \textit{weak Calabi-Yau} if $K_X\sim_\mathbb{Q} 0$ and $q^\circ(X)=0$ (cf.~\cite[Definition 2.9]{NZ10}).

\begin{cor}[{cf.~\cite[Theorem 0.4]{CCP08}}]\label{main_Goren_ter_3fold}
Let $X$ be a $\mathbb{Q}$-factorial Gorestein terminal projective threefold, and $L_X$ a strictly nef $\mathbb{Q}$-divisor on $X$.
Then $K_X+tL_X$ is ample for sufficiently large $t\gg 1$ unless $X$ is weak Calabi-Yau and $L_X\cdot c_2(X)=0$ where $c_2(X)$ is the birational second Chern class on $X$ (cf.~\cite{SW94}).
\end{cor}

\paragraph{Acknowledgement}
The first author is grateful to Professor Xiaokui Yang for bringing to his attention the study of varieties strictly nef anticanonical divisors. He is partially supported by the by the National Natural Science Foundation of China. 
The second author would like to 
thank Professor De-Qi Zhang for inspiring discussions.
He would also like to thank Professor Vladimir Lazi\'c
for pointing out \cite[Corollary 1.8]{HL20} to him. 
He is  supported by President's Scholarships of NUS.

\section{Preliminary results}
\label{sec_pre}
Throughout this paper, we refer to \cite[Chapter 2]{KM98} for notations and terminologies on different kinds of singularities.
By a \textit{projective klt (resp. canonical) pair} $(X,\Delta)$, we mean that $X$ is a normal projective variety, $\Delta\geqslant 0$ is an effective $\mathbb{Q}$-divisor such that $K_X+\Delta$ is $\mathbb{Q}$-Cartier, and  $(X,\Delta)$ has only klt (resp. canonical) singularities.

In this section, let us prepare some preliminary results, which will be crucially used in the later sections. 
Some of the results have been proved in \cite[Section 5]{LOWYZ21}; however,  for the convenience of readers and us, we recall them here.  
As an application, we give an alternative proof of {\hyperref[main_thm_surface]{Proposition  \ref*{main_thm_surface}}}.


\begin{lemme}[{cf.~\cite[Lemma 5.1]{LOWYZ21}, \cite[Lemma 1.1]{Ser95}}]\label{lem_strict_nef_k}
Let $(X,\Delta)$ be a projective klt pair, and $L_X$  a strictly nef $\mathbb{Q}$-divisor on $X$.
Then $(K_X+\Delta)+tL_X$ is  strictly nef for every $t>2m\dim X$ where $m$ is the Cartier index of $L_X$.
\end{lemme}

Note that, when $\dim X=2$, the conclusion of  {\hyperref[lem_strict_nef_k]{Lemma  \ref*{lem_strict_nef_k}}} holds with a more optimal lower bound $t>3$ (cf.~\cite[Proof of Lemma 5.1]{LOWYZ21} and \cite[Proposition 3.8]{Fuj12}).

\begin{lemme}[{cf.~\cite[Lemma 5.2]{LOWYZ21}}]\label{lem-big-ample}
Let $(X,\Delta)$ be a projective klt pair, 
and $L_X$  a strictly nef $\mathbb{Q}$-divisor on $X$.	
Suppose  $a(K_X+\Delta)+bL_X$ is big for some $a,b\geqslant 0$.
Then $(K_X+\Delta)+tL_X$ is ample for sufficiently large $t$. 
\end{lemme}

Let  $\overline{\textup{NE}}(X)$ (resp. $\overline{\textup{ME}}(X)$) be the \textit{Mori cone} (resp. \textit{movable cone}) of $X$ (cf.~\cite{BDPP13}).
The next lemma characterizes the situation when $K+tL$ is not big.
Note that, the first part of the lemma has been proved in \cite[Lemma 5.3]{LOWYZ21}.
For the organization, we include a complete proof here.
\begin{lemme}[{cf.~\cite[Lemma 5.3]{LOWYZ21}}]\label{lem-not-big-some}
Let $(X,\Delta)$ be a projective klt  pair of dimension $n$, 
and $L_X$  a strictly nef $\mathbb{Q}$-divisor on $X$.	
Then $(K_X+\Delta)+uL_X$ is not big for some rational number $u>2mn$ with $m$  the Cartier index of $L_X$, if and only if
$(K_X+\Delta)^i\cdot L_X^{n-i}=0$ for any $0\leqslant i\leqslant n$. 	
Moreover, if one of the above two equivalent conditions holds, then there exists a class $0\neq \alpha\in\overline{\textup{ME}}(X)$ such that 
$(K_X+\Delta)\cdot\alpha=L_X\cdot\alpha=0$. 
\end{lemme}

\begin{proof}
Suppose first that $(K_X+\Delta)+uL_X$ is not big for some rational number $u>2mn$. 
Set $u':=u-2mn>0$, $D_1:=(K_X+\Delta)+(\frac{u'}{2}+2mn)L_X$ and $D_2:=\frac{u'}{2}L_X$.	
Then $(K_X+\Delta+uL_X)^n=(D_1+D_2)^n=0$.
Since both $D_1$ and $D_2$ are nef (cf.~{\hyperref[lem_strict_nef_k]{Lemma  \ref*{lem_strict_nef_k}}}), the vanishing $D_1^i\cdot D_2^{n-i}=0$ for every $0\leqslant i\leqslant n$  yields $(K_X+\Delta)^i\cdot L_X^{n-i}=0$ for any $0\leqslant i\leqslant n$. 
Conversely, if $(K_X+\Delta)^i\cdot L_X^{n-i}=0$ for any $0\leqslant i\leqslant n$, then $((K_X+\Delta)+uL_X)^n=0$ for any $u$; in this case, $(K_X+\Delta)+uL_X$ being not big for any $u>2mn$ follows from \cite[Proposition 2.61, p.~68]{KM98} and {\hyperref[lem_strict_nef_k]{Lemma  \ref*{lem_strict_nef_k}}}. 
Therefore, the first part of our lemma is proved.

Assume that the nef divisor $K_X+\Delta+uL_X$ is not big for some $u>2mn$. 
Then, there exists a class $\alpha\in\overline{\textup{ME}}(X)$ such that $(K_X+\Delta+uL_X)\cdot\alpha=0$ (cf.~\cite[Theorem 2.2 and the remarks therein]{BDPP13}). 
Suppose that $L_X\cdot\alpha\neq 0$.  
Then $L_X\cdot\alpha>0$ and thus $(K_X+\Delta)\cdot\alpha<0$.
By the cone theorem (cf.~\cite[Theorem 3.7, p.~76]{KM98}), we have 
$\alpha=M+\sum a_iC_i$ 
where $M$ is a class lying in $\overline{\textup{NE}}(X)_{(K_X+\Delta)\geqslant 0}$, $a_i>0$, and the $C_i$ are extremal rational curves with $0<-(K_X+\Delta)\cdot C_i\leqslant 2n$.
Now, for each $i$, the intersection $(K_X+\Delta+uL_X)\cdot C_i>0$, a contradiction to the choice of $\alpha$.
So we have $(K_X+\Delta)\cdot \alpha=L_X\cdot\alpha=0$.
\end{proof}

Now we give an alternative proof of {\hyperref[main_thm_surface]{Proposition  \ref*{main_thm_surface}}} (cf.~\cite[Corollary 1.8]{HL20}) by applying a recent result on \textit{almost strictly nef} divisors on surfaces (cf.~\cite[Theorem 20]{Cha20} and \cite[Proposition 2.3]{CCP08}). 
Let us first recall the following definition. 
\begin{defn}[{cf.~\cite[Definition 1.1]{CCP08}}]\label{defn-almost-sn}
A $\mathbb{Q}$-Cartier divisor $L$ on a normal projective variety $X$ is called  \textit{almost strictly nef}, if there exist a birational morphism $\pi:X\to Y$ to a  normal projective variety $Y$ and a strictly nef $\mathbb{Q}$-divisor $L_Y$ on $Y$ such that $L=\pi^*L_Y$.  
\end{defn}
One of the motivation to introduce almost strictly nef divisors is to study the  non-uniruled case of {\hyperref[main-conj-singular-arbitrary]{Question \ref*{main-conj-singular-arbitrary}}}  via the descending of Iitaka fibrations, in which case, the pullback of a strictly nef divisor to a higher model which resolves the indeterminacy is almost strictly nef  (cf.~{\hyperref[thm_higher-dim-kappa>=1]{Proposition  \ref*{thm_higher-dim-kappa>=1}}} and \cite[Theorem 2.6]{CCP08}). 

\begin{proof}[\textup{\textbf{Proof of {\hyperref[main_thm_surface]{Proposition \ref*{main_thm_surface}}}}}]
Taking a minimal resolution $\pi:\widetilde{X}\to X$, we see that $L_{\widetilde{X}}:=\pi^*L_X$ is almost strictly nef. 
By \cite[Theorem 26]{Cha20},  $K_{\widetilde{X}}+tL_{\widetilde{X}}$ is big for all $t>3$.
Then its push-down $K_X+tL_X$	 is big (as a Weil divisor) for  all $t>3$ (cf.~\cite[Lemma 4.10 (i\!i\!i) and its proof]{FKL16}), i.e., $K_X+tL_X$ lies in the interior part of the closure of the cone generated by effective Weil divisors on $X$ for all $t>3$.  
So $K_X+\Delta+tL_X$ is big (as a Cartier divisor) for all $t>3$.
Since $L_X$ is a strictly nef Cartier divisor, by  {\hyperref[lem-big-ample]{Lemma  \ref*{lem-big-ample}}} and the remark after it, we see that for $t>3$,  $K_X+\Delta+tL_X$ is strictly nef and $K_X+\Delta+2tL_X=2(K_X+\Delta+tL_X)-(K_X+\Delta)$ is  nef and big.  
Applying the base-point-free theorem for $K_X+\Delta+tL_X$ (cf.~\cite[Theorem 3.3, p. 75]{KM98}), 
 our theorem follows from the strict nefness of $K_X+\Delta+tL_X$ for $t>3$ (cf. {\hyperref[lem-big-ample]{Lemma  \ref*{lem-big-ample}}} and the remark after it).   
\end{proof}

The following proposition was proved in \cite[Proposition 5.4]{LOWYZ21} which slightly generalizes {\hyperref[main_thm_surface]{Proposition \ref*{main_thm_surface}}}  to the (not necessarily normal)  $\mathbb{Q}$-Gorenstein surface case (cf.~\cite[Conjecture 1.3]{CCP08}).  
Recall that, for a $\mathbb{Q}$-factorial normal projective variety $X$ and a prime divisor $S\subseteq X$,  the \textit{canonical divisor} $K_S\in\textup{Pic}(S)\otimes\mathbb{Q}$ is defined by  
$K_S:=\frac{1}{m}(mK_X+mS)|_S$, where $m\in\mathbb{N}$ is the smallest positive integer such that both $mK_X$ and $mS$ are Cartier divisors on $X$.

\begin{prop}[{cf.~\cite[Proposition 5.4]{LOWYZ21}}]\label{prop_Q-Goren_surface}
Let $(X,\Delta)$ be a $\mathbb{Q}$-factorial klt threefold pair,  $S$ a prime divisor on $X$, and  $L_S$ a strictly nef  divisor (resp. almost strictly nef) on $S$.
Then $K_S+tL_S$ is ample (resp. big) for sufficiently large $t$.
\end{prop}

As an application, we  extend \cite[Proposition 3.1]{Ser95} to the following singular setting, 
which paves a possible way to consider {\hyperref[main-conj-singular-arbitrary]{Question \ref*{main-conj-singular-arbitrary}}} and prove the ampleness of $K+tL$ for $t\gg 1$.

\begin{prop}\label{prop-q-effective}
Let $X$ be a $\mathbb{Q}$-factorial normal projective threefold with only klt singularities,  and  $L_X$ a strictly nef $\mathbb{Q}$-Cartier divisor on $X$.
Suppose that there is a non-zero effective divisor numerically equivalent to $aL_X+bK_X$ for some  $a$ and $b$. 
Then $K_X+tL_X$ is ample for  $t\gg 1$.
\end{prop}

\begin{proof}
Suppose  the contrary that $K_X+tL_X$ is not ample for some (and hence all) $t\gg 1$. 
By {\hyperref[lem_strict_nef_k]{Lemma  \ref*{lem_strict_nef_k}}} and {\hyperref[lem-big-ample]{Lemma  \ref*{lem-big-ample}}},  $K_X+tL_X$ is not big (but nef) for some  $t>6m$ with $m$ the Cartier index of $L_X$. 
Then {\hyperref[lem-not-big-some]{Lemma  \ref*{lem-not-big-some}}} gives us that
$$K_X^3=K_X^2\cdot L_X=K_X\cdot L_X^2=L_X^3=0.$$
Pick an effective divisor $\sum n_iF_i\equiv aL_X+bK_X$.
Then  $(K_X+tL_X)^2\cdot (aL_X+bK_X)=0$.
Since $K_X+tL_X$ is nef for all $t\gg 1$,  
we have  $F_i\cdot L_X^2=F_i\cdot K_X\cdot L_X=F_i\cdot K_X^2=0$ for each prime component $F_i$. 
By {\hyperref[prop_Q-Goren_surface]{Proposition  \ref*{prop_Q-Goren_surface}}}, $K_{F_1}+rL_X|_{F_1}$ is ample for $r\gg 1$.  
Therefore, we have
$$0<L_X|_{F_1}\cdot (K_{F_1}+rL_X|_{F_1})=F_1\cdot L_X\cdot (K_X+F_1+rL_X)=F_1^2\cdot L_X.$$
Now that $F_i$ and $F_j$ are distinct prime divisors, we deduce the following
$$0=F_1\cdot (aL_X+bK_X)\cdot L_X=F_1\cdot\sum  n_iF_i\cdot L_X\geqslant n_1F_1^2L_X>0,$$
a contradiction. 
So $K_X+tL_X$ is ample for  $t\gg 1$ and our proposition is thus proved.
\end{proof}

In what follows, we extend  {\hyperref[prop-q-effective]{Proposition  \ref*{prop-q-effective}}} to the  non-$\mathbb{Q}$-factorial cases: either (i) $L_X$ is num-effective, i.e., numerically equivalent to a non-zero effective divisor (which has been proved in \cite[Proposition 5.5]{LOWYZ21}), or (ii) $\Delta=0$ and $X$ has only canonical singularities (cf.~{\hyperref[prop-klt-eff-ample]{Proposition \ref*{prop-klt-eff-ample}}}). 
However, we are still not able to deal with the general non-$\mathbb{Q}$-factorial case without assuming any further conditions.

\begin{prop}[{cf.~\cite[Proposition 5.5]{LOWYZ21}}]\label{prop-klt-eff-ample}
Let $(X,\Delta)$ be a projective klt threefold pair.
Let $L_X$ be a strictly nef $\mathbb{Q}$-divisor on $X$.
Suppose that there is a non-zero effective divisor which is numerically equivalent to $aL_X+b(K_X+\Delta)$.
Suppose further that either (i) $b=0$, or (ii) $\Delta=0$ and $X$ has only canonical singularities. 
Then $K_X+\Delta+tL_X$ is ample for sufficiently large $t$. 
\end{prop}

\begin{proof}
Suppose the contrary that $K_X+\Delta+tL_X$ is not ample for one (and hence any) $t\gg 1$.
By {\hyperref[lem-big-ample]{Lemma \ref*{lem-big-ample}}} and {\hyperref[lem-not-big-some]{Lemma \ref*{lem-not-big-some}}} we have 
\begin{equation}\label{eq_int-nb=0}
(K_X+\Delta)^3=(K_X+\Delta)\cdot L_X^2=(K_X+\Delta)\cdot L_X=L_X^3=0.
\end{equation}
Let us take a $\mathbb{Q}$-factorialization $\pi:(Y,\Gamma)\to (X,\Delta)$, which is a small birational morphism (cf.~\cite[Corollary 1.37, pp.~29-30]{Kollar13}) and let $L_Y:=\pi^*L_X$ which is  nef but not big. 
Then $L_Y^3=0$.

Under the assumption  (i), i.e., $L_X$ is numerically equivalent to a non-zero effective divisor, our proposition  follows from \cite[Proposition 5.5]{LOWYZ21}. 
Therefore, we assume the condition (ii), i.e., $X$ has only canonical singularities (and thus $K_X$ is $\mathbb{Q}$-Cartier) and $aL_X+bK_X$ is numerically equivalent to a non-zero effective divisor.
In this case, it follows from \cite{Kaw88} that we can take $\pi$ to be crepant, i.e., $K_Y=\pi^*K_X$.
Then $aL_Y+bK_Y$ is also numerically equivalent to a non-zero effective divisor $\sum n_iF_i$.
By \eqref{eq_int-nb=0} and the projection formula, we have $L_Y^i\cdot K_Y^{3-i}=L_X^i\cdot K_X^{3-i}=0$ for every $i$; hence by  {\hyperref[prop_Q-Goren_surface]{Proposition \ref*{prop_Q-Goren_surface}}}, we see that  $K_{F_i}+tL_Y|_{F_i}$ is big for $t\gg 1$. 
Since $L_Y^3=0$, we have $(L_Y|_{F_i})^2=L_Y^2\cdot F_i=0$ for every $i$ due to the nefness of $L_Y$.
By the Hodge index theorem, 
$$0<L_Y|_{F_i}\cdot (K_{F_i}+tL_Y|_{F_i})=L_Y\cdot (K_Y+F_i)\cdot F_i.$$
Since $(K_Y+tL_Y)^2\cdot (aL_Y+bK_Y)=0$ and $K_Y+tL_Y=\pi^*(K_X+tL_X)$ is nef for $t\gg 1$, we have $(K_Y+tL_Y)^2\cdot F_i=0$ for any $t\gg 1$.
This further implies that $L_Y\cdot K_Y\cdot F_i=0$, and therefore $L_Y\cdot F_i^2>0$ for every $i$ by the above inequality.
But then, we get again the following contradiction:
$$0=aL_Y^2\cdot F_1=L_Y\cdot(aL_Y+bK_Y)\cdot F_1=L_Y\cdot\sum n_iF_i\cdot F_1\geqslant L_Y\cdot F_1^2>0.$$
Hence, our proposition is proved under the condition (ii).
\end{proof}

To end up this section, we extend a formula on conic bundles to the singular case.
Recall that a \textit{contraction} of a  normal projective variety $X$ is a surjective morphism with connected fibres.
A contraction $\pi:X\to S$ is said to be \textit{elementary} if the relative Picard number $\rho(X/S)=1$.
Recall also that a \textit{conic bundle} is a flat morphism between nonsingular varieties such that the generic fibre is an irreducible rational curve. 

\begin{lemme}[{cf.~\cite[4.11]{Miy81}, {\hyperref[rem_composition_conic]{Remark  \ref*{rem_composition_conic}}}}]\label{lem_conic_miy}
Let $X$ be a normal projective threefold with at worst isolated klt singularities.
Suppose that $\pi:X\to S$ is an elementary  $K_X$-negative contraction  onto a normal projective surface $S$.
Denote by $D_1$ the one-dimensional part of the discriminant locus of $\pi$ (over which, $\pi$ is not smooth).
Then $\pi_*K_X^2\equiv -(4K_S+D_1)$.
\end{lemme}

\begin{proof}
Since $\rho(X/S)=1$, by the  canonical bundle formula (cf.\,e.g.~\cite[Theorem 0.2]{Amb05}), there exists some  $\Delta_S\geqslant 0$ on $S$ such that $(S,\Delta_S)$ is klt and thus $S$ is $\mathbb{Q}$-factorial (cf.~e.g.~\cite[Proposition 4.11]{KM98}), and
we only need to check the equality  on very ample curves  on $S$. 
By the cone theorem (cf.~\cite[Theorem 3.7]{KM98}), $\pi$ is equi-dimensional.
By \cite[Theorem 5.10]{KM98}, $X$ is Cohen-Macaulay.
Let $D_0:=(\textup{Sing}~S)\cup\pi(\textup{Sing}~X)$.
Then, $\pi|_{X\backslash\pi^{-1}(D_0)}$ is flat 
and thus $\pi|_{X\backslash\pi^{-1}(D_0)}$ is a usual conic bundle.  
Let $T$ be a very ample curve on $S$ which can be assumed to be smooth, avoid $D_0$, and intersect with $D_1$ transversally. 
Let $F:=\pi^{-1}(T)=\pi^*(T)$, which is smooth.
Then 
$$K_X^2\cdot F=(K_X|_F)^2=(K_F-F|_F)^2=K_F^2-2K_F\cdot (F|_F)=K_F^2+4T^2.$$ 
Note that the last equality is due to the adjunction.
Since $F$ is a (not necessarily minimal) ruled surface over $T$ with $D_1
\cdot T$ degenerate  (reducible) fibres, we have
$$K_F^2=-4(K_S\cdot T+T^2)-D_1\cdot T.$$ 
As a result, we have $\pi_*(K_X^2)\cdot T=K_X^2\cdot F=-(4K_S+D_1)\cdot T$.
So our lemma is proved.
\end{proof}

\section{Kodaira fibrations,  Proof of \texorpdfstring{{\hyperref[main_theorem_Goren_ter_3fold]{Theorem \ref*{main_theorem_Goren_ter_3fold} (1)}}}{text}}

In this section, we study the case when $X$ has positive Kodaira dimension with the main results {\hyperref[kappa>=1]{Theorem \ref*{kappa>=1}}} and 
{\hyperref[thm_higher-dim-kappa>=1]{Proposition  \ref*{thm_higher-dim-kappa>=1}}}. 

First, {\hyperref[kappa>=1]{Theorem \ref*{kappa>=1}}} below is a more general (pair) version of {\hyperref[main_theorem_Goren_ter_3fold]{Theorem   \ref*{main_theorem_Goren_ter_3fold} (1)}}.

\begin{thm}[{cf.~{\hyperref[main_theorem_Goren_ter_3fold]{Theorem   \ref*{main_theorem_Goren_ter_3fold} (1)}}}]\label{kappa>=1}
Let $(X,\Delta)$ be a projective klt threefold pair, and $L_X$ a strictly nef $\mathbb{Q}$-Cartier divisor on $X$.
Suppose that the log Kodaira dimension $\kappa(X,K_X+\Delta)\geqslant 1$.
Then $K_X+\Delta+tL_X$ is ample for sufficiently large $t$.
\end{thm}

\begin{proof}
Since $L_X$ is strictly nef, by {\hyperref[lem_strict_nef_k]{Lemma \ref*{lem_strict_nef_k}}}, $K_X+\Delta+tL_X$ is strictly nef for $t\gg 1$.
Clearly, $K_X+\Delta$ is pseudo-effective.
Hence, it follows from \cite[Corollary D]{LP20A} that $K_X+\Delta+tL_X$ is numerically equivalent to a semiample $\mathbb{Q}$-divisor $M$.
Then with $M$ replaced by a multiple, it defines a morphism $\sigma:X\to Y$ with $M=\sigma^*\mathcal{O}_Y(1)$.
Therefore, the strict nefness of $M$ together with the projection formula implies that $\sigma$ is a finite morphism and thus $M$ is  ample.
So our theorem is proved.
\end{proof}

As a higher dimensional analogue of {\hyperref[kappa>=1]{Theorem  \ref*{kappa>=1}}},
we show the following proposition by further requiring  the pair $(X,\Delta)$ having only canonical singularities.
Note that, this is also an extension of \cite[Theorem 2.6]{CCP08} and has its independent interests.
\begin{prop}[{cf.~{\hyperref[kappa>=1]{Theorem  \ref*{kappa>=1}}} and \cite[Theorem 2.6]{CCP08}}]\label{thm_higher-dim-kappa>=1}
Let $(X,\Delta)$ be a projective canonical pair, and $L_X$ a strictly nef $\mathbb{Q}$-divisor on $X$.
Suppose  the Iitaka dimension $\kappa(a(K_X+\Delta)+bL_X)\geqslant\dim X-2$ for some $a,b\geqslant 0$.
Then $(K_X+\Delta)+tL_X$ is ample for sufficiently large $t$.
\end{prop}

\begin{proof}
If $\dim X\leqslant 2$, then our theorem follows from {\hyperref[main_thm_surface]{Proposition \ref*{main_thm_surface}}}.
So we may assume $\dim X\geqslant 3$.
By {\hyperref[lem_strict_nef_k]{Lemma  \ref*{lem_strict_nef_k}}}, $K_X+\Delta+tL_X$ is strictly nef for $t\gg 1$. 
With $b$ replaced by $b/a$ (if $a\neq 0$), we may further assume  $0\leqslant a\leqslant 1$.
Let $\pi:X\dashrightarrow Y$ be the Iitaka fibration of $a(K_X+\Delta)+bL_X$ (with $\dim Y\geqslant 1$).
Resolving the indeterminacy of $\pi$ and the singularities of $X$, we get the induced morphisms $\pi_1:\widetilde{X}\to X$ and $\pi_2:\widetilde{X}\to Y$ such that  $\widetilde{X}$ is smooth and
\begin{equation}\label{equ_1}
	K_{\widetilde{X}}+\widetilde{\Delta}=\pi_1^*(K_X+\Delta)+\sum a_iE_i
\end{equation}
with $\widetilde{\Delta}:=(\pi_1)_*^{-1}\Delta$ and $a_i\geqslant 0$, noting that $(X,\Delta)$ is a projective canonical pair.
Moreover, there exists an ample $\mathbb{Q}$-divisor $H$ on $Y$ such that
\begin{equation}\label{equ_2}
\pi_1^*(a(K_X+\Delta)+bL_X)\sim_{\mathbb{Q}}\pi_2^*H+E_0.	
\end{equation}
Here, some multiple $nE_0$ is the fixed component of $|n\pi_1^*(a(K_X+\Delta)+bL_X)|$ (which defines $\pi_2$) and hence $E_0\ge 0$. 
Let $L_{\widetilde{X}}:=\pi_1^*L_X$, being almost strictly nef (cf.~{\hyperref[defn-almost-sn]{Definition \ref*{defn-almost-sn}}}) 
on $\widetilde{X}$.
Then for a general fibre $F$ of $\pi_2$, the restriction $L_{\widetilde{X}}|_F$ is almost strictly nef.

If $\kappa(a(K_X+\Delta)+bL_X)=\dim X$, then our theorem follows from {\hyperref[lem-big-ample]{Lemma  \ref*{lem-big-ample}}}.
If $\kappa(a(K_X+\Delta)+bL_X)=\dim X-1$, then  $L_{\widetilde{X}}|_F$ is clearly ample.
If $\kappa(a(K_X+\Delta)+bL_X)=\dim X-2$, then $\dim F=2$; in this case, it follows from \cite[Theorem 26]{Cha20} that $K_F+tL_{\widetilde{X}}|_F$ is big for $t\gg 1$ and hence $(K_{\widetilde{X}}+\widetilde{\Delta}+tL_{\widetilde{X}})|_{F}$ is big for $t\gg 1$.
Let us pick the ample $\mathbb{Q}$-divisor $H$ on $Y$ in {\hyperref[equ_2]{Equation  (\ref*{equ_2})}}.

In each case, we see that 
$K_{\widetilde{X}}+\widetilde{\Delta}+tL_{\widetilde{X}}+\pi_2^*H$
is big for  $t\gg 1$: fixing one $t$ such that $(K_{\widetilde{X}}+\widetilde{\Delta}+tL_{\widetilde{X}})|_{F}$ is big, we see that  $A_k:=(K_{\widetilde{X}}+\widetilde{\Delta}+tL_{\widetilde{X}})+k\pi_2^*H$ is big for $k\gg 1$ (cf.~e.g. \cite[Lemma 3.23]{KM98}).
Since $K_{\widetilde{X}}+\widetilde{\Delta}+tL_{\widetilde{X}}=\pi_1^*(K_X+\Delta+tL_X)+\sum a_iE_i$ is   pseudo-effective, our $A_1=K_{\widetilde{X}}+\widetilde{\Delta}+tL_{\widetilde{X}}+\pi_2^*H=\frac{1}{k}(A_k+(k-1)(K_{\widetilde{X}}+\widetilde{\Delta}+tL_{\widetilde{X}}))$ is big.

Set $N:=(1-a)(K_{\widetilde{X}}+\widetilde{\Delta})+(t-b)L_{\widetilde{X}}$, which is pseudo-effective for sufficiently large $t\gg 1$ (cf.~{\hyperref[lem_strict_nef_k]{Lemma  \ref*{lem_strict_nef_k}}} and {\hyperref[equ_1]{Equation  (\ref*{equ_1})}}).
Following from {\hyperref[equ_1]{Equation  (\ref*{equ_1})}} and {\hyperref[equ_2]{Equation  (\ref*{equ_2})}}, we have
\begin{align*}
2(K_{\widetilde{X}}+\widetilde{\Delta}+tL_{\widetilde{X}})&=A_1+(a(K_{\widetilde{X}}+\widetilde{\Delta})+bL_{\widetilde{X}}-\pi_2^*H)+N\\
&=
A_1+\pi_1^*(a(K_X+\Delta)+bL_X)+\sum aa_iE_i-\pi_2^*H+N\\
&\sim_{\mathbb{Q}} A_1+E_0+\sum aa_iE_i+N(\ge A_1+N),
\end{align*}
which is big for sufficiently large $t\gg 1$.
Thus, the push-forward $K_X+\Delta+tL_X$ is also big for $t\gg1$ (cf.~\cite[Lemma 4.10(i\!i\!i) and its proof]{FKL16}).
So our theorem follows from {\hyperref[lem-big-ample]{Lemma  \ref*{lem-big-ample}}}.
\end{proof}

\section{Non-trivial Albanese morphism,  Proof of \texorpdfstring{\hyperref[main_theorem_Goren_ter_3fold]{Theorem   \ref*{main_theorem_Goren_ter_3fold} (2)}}{text}}

In this section, we study 
{\hyperref[main-conj-singular-arbitrary]{Question    \ref*{main-conj-singular-arbitrary}}} for the case when the augmented irregularity is positive  with the main results {\hyperref[thm_alb]{Theorem   \ref*{thm_alb}}} and {\hyperref[coro_q0>0]{Corollary   \ref*{coro_q0>0}}}; in particular, we prove {\hyperref[main_theorem_Goren_ter_3fold]{Theorem   \ref*{main_theorem_Goren_ter_3fold} (2)}}. 

First, the following theorem is a singular version of \cite[Theorem 3.1]{CCP08}.  
The proof is almost the same after we replace  \cite[Theorem 0.3, Lemma 1.5, Propositions 1.7, 1.8]{CCP08} by 
\cite[Theorem A]{LOWYZ21}, {\hyperref[lem-big-ample]{Lemma  \ref*{lem-big-ample}}},  {\hyperref[prop_Q-Goren_surface]{Proposition  \ref*{prop_Q-Goren_surface}}} and {\hyperref[prop-q-effective]{Proposition  \ref*{prop-q-effective}}}, respectively. 
For readers' convenience, we  include a more detailed proof here. 
\begin{thm}[{cf.~\cite[Theorem 3.1]{CCP08}, \cite[Theorem 8.1]{LP20B}}]\label{thm_alb}
Let $(X,\Delta)$ be a   projective klt threefold pair, and $L_X$ a  strictly nef $\mathbb{Q}$-divisor on $X$.
Suppose there exists a non-constant morphism $\pi:X\to A$ to an abelian variety (this is the case when $q(X)>0$).
Then $K_X+\Delta+tL_X$ is numerically equivalent to a non-zero effective divisor for  some $t\gg 1$. 
If $X$ is further assumed to be either $\mathbb{Q}$-factorial, or have only canonical singularities, 
then $K_X+tL_X$ is ample for $t\gg 1$.
\end{thm}

\begin{proof}
Replacing $L_X$ by a multiple, we may assume that $L_X$ is Cartier.
Let $m$ be the Cartier index of $K_X+\Delta$, $D_t:=2m(K_X+\Delta+tL_X)$ and $\mathcal{F}_t:=\pi_*\mathcal{O}_X(D_t)$. 
Fix one $t>6m$ such that $D_t$ is strictly nef and  $D_t|_{F}$ is ample for a general fibre $F$ of $\pi$ (cf.~{\hyperref[main_thm_surface]{Proposition \ref*{main_thm_surface}}} and \cite[Theorem 1.37]{KM98}). 
Then it is easy to verify that $D_t$ is $\pi$-big  (cf.~\cite[Definition 3.22]{KM98}). 
Since $D_t-(K_X+\Delta)=(2m-1)(K_X+\Delta+tL_X)+tL_X$ is  $\pi$-big and $\pi$-nef, by the relative base-point-free theorem (cf.~\cite[Theorem 3.24]{KM98}), $nD_t$ is $\pi$-free  for all $n\gg 1$.
Hence, with $m$ replaced by a multiple, we may further assume that $D_t$ is $\pi$-free (i.e., $\pi^*\mathcal{F}_t\to \mathcal{O}_X(D_t)$ is surjective), and thus $\mathcal{F}_t$ is a non-zero sheaf.

Let $\hat{A}=\textup{Pic}^0(A)$ be the dual abelian variety and $\mathcal{P}$  the normalized Poincar\'e line bundle (in the sense that both $\mathcal{P}|_{A\times\{\hat{0}\}}$ and $\mathcal{P}|_{\{0\}\times\hat{A}}$ are trivial).
Denote by $p_A$ and $p_{\hat{A}}$ the projections of $A\times\hat{A}$ onto $A$ and $\hat{A}$.
For any ample line bundle $\hat{M}$ on $\hat{A}$, we define the isogeny $\phi_{\hat{M}}:\hat{A}\to A$  by $\phi_{\hat{M}}(\hat{a})=t_{\hat{a}}^*\hat{M}^{\vee}\otimes\hat{M}$, and
let $M:=(p_{A})_*(p_{\hat{A}}^*\hat{M}\otimes\mathcal{P})$ be the vector bundle on $A$.
Recall that $\phi_{\hat{M}}^*(M^{\vee})\simeq\hat{M}^{\oplus\dimcoh^0(\hat{M})}$ (cf.~\cite[Proposition 3.11 (1)]{Muk81}). 
\begin{claim}\label{claim_fm}
$H^i(A,\mathcal{F}_t\otimes M^{\vee})=0$ for any  $t\gg1$ and any ample line bundle $\hat{M}$ on $\hat{A}$. 
\end{claim}
Suppose the claim for the time being.
Then it follows from \cite[Theorem 1.2 and Corollary 3.2]{Hac04} that
we have a chain of inclusions
$V^0(\mathcal{F}_t)\supseteq V^1(\mathcal{F}_t)\supseteq\cdots\supseteq V^n(\mathcal{F}_t)$, 
where $V^i(\mathcal{F}_t):=\{P\in\textup{Pic}^0(A)~|~\dimcoh^i(A,\mathcal{F}_t\otimes P)\neq 0\}$.
Recall that $\mathcal{F}_t$ is a  non-zero sheaf by the choice of our $m$. 
If  $V^i(\mathcal{F}_t)=\emptyset$ for all $i$, then the Fourier-Mukai transform of $\mathcal{F}_t$ is zero, noting that $\mathcal{P}|_{A\times\{P\}}\cong P$ and 
$\RDer^i\!(p_{\hat{A}})_*(p_A^*\mathcal{F}_t\otimes\mathcal{P})_P=\Coh^i(A,\mathcal{F}_t\otimes P)$ 
for any $P\in\textup{Pic}^0(X)$.  
By \cite[Theorem 2.2]{Muk81},  $\mathcal{F}_t$ is a zero sheaf, which is absurd. 
So $V^0(\mathcal{F}_t)\neq\emptyset$ and thus
$\dimcoh^0(X,D_t+\pi^*P)\neq 0$ for some $P\in\textup{Pic}^0(A)$.
If $D_t+\pi^*P\sim 0$, then $-(K_X+\Delta)$ is strictly nef; hence $X$ is rationally connected and  $q(X)=0$ by \cite[Theorem A]{LOWYZ21},  which contradicts the existence of $\pi$. 
So  $D_t\equiv D_t+\pi^*P\sim E$ for some non-zero effective divisor $E$, 
and the first part of our theorem is proved.

If $X$ is further assumed to be  $\mathbb{Q}$-factorial (resp. have only canonical singularities), then we apply the first part of our theorem for the klt (resp. canonical) pair $(X,0)$, and thus $K_X+tL_X$ is numerically equivalent to a non-zero effective divisor; hence, the second part of our theorem 
 follows from {\hyperref[prop-q-effective]{Proposition  \ref*{prop-q-effective}}} (resp. {\hyperref[prop-klt-eff-ample]{Proposition  \ref*{prop-klt-eff-ample}}}). 
 
\par \vskip 1pc \noindent
\textbf{Proof of {\hyperref[claim_fm]{Claim \ref*{claim_fm}}}.}
Fix an ample divisor $H$ on $A$. 
Since $D_{t}$ is  $\pi$-big, our $D_t+\pi^*H$ is big, noting that $D_t$ is not only $\pi$-big, but also nef (and hence pseudo-effective) (cf.~\cite[Lemma 3.23]{KM98}).
Therefore, $D_t+\pi^*(H+P)$ is nef and big for any $P\in\textup{Pic}^0(A)$. 
Since $(X,\Delta)$ has at worst klt singularities and 
$$D_t+\pi^*(H+P)=K_X+\Delta+\frac{2m-1}{2m}(D_t+\pi^*(H+P))+\frac{1}{2m}\pi^*(H+P)+tL_X=:K_X+\Delta+\widetilde{D}$$
with $\widetilde{D}$ being a nef and big $\mathbb{Q}$-Cartier Weil divisor, it follows from  Kawamata-Viehweg Vanishing Theorem (cf. e.g. \cite[Theorem 2.70]{KM98}) that 
$$\Coh^j(A,\mathcal{F}_t\otimes H\otimes P)\simeq\Coh^j(X,D_t+\pi^*H+\pi^*P)\simeq0$$
for all $j\geqslant 1$ (recall that $\mathcal{F}_t:=\pi_*\scrO_X(D_t)$). 
The first equality is due to $\RDer^j\pi_*(D_t+\pi^*H+\pi^*P)=0$ for all $j\geqslant 1$ (cf.~\cite[Remark 1-2-6]{KMM87}).
Following the same terminology as in \cite[Definition 2.3]{Muk81}, for any ample line bundle $H$ on $A$, we say that \textit{the index theorem of degree zero holds for the sheaf $\mathcal{F}_t\otimes H$} (which is also abbreviated to ``$IT^0$'' in \cite[Proof of Theorem 3.1]{CCP08}).

Let $\hat{M}$ be any ample line bundle on $\hat{A}$, and $\phi_{\hat{M}}:\hat{A}\to A$  the isogeny defined  in the beginning of the proof. 
Let $\hat{\pi}:\hat{X}:=X\times_A\hat{A}\to \hat{A}$ be the base change with the induced map $\varphi:\hat{X}\to X$ being \'etale.
Then we have $K_{\hat{X}}=\varphi^*K_X$ and $L_{\hat{X}}:=\varphi^*L_X$ is strictly nef on $\hat{X}$.
Let $\mathcal{G}_t:=\hat{\pi}_*\varphi^*D_t$.
With the same argument as above, we see that, for any ample line bundle $\hat{M}$ on $\hat{A}$, the index theorem of degree zero holds for the sheaf $\mathcal{G}_t\otimes\hat{M}$. 
Since $\phi_{\hat{M}}$ is flat and $\pi$ is projective, we have
\begin{align*}
&~\phi_{\hat{M}}^*(\mathcal{F}_t\otimes M^{\vee})=\phi_{\hat{M}}^*(\pi_*D_t\otimes M^{\vee})=\phi_{\hat{M}}^*\pi_*(D_t\otimes \pi^*M^{\vee})=\hat{\pi}_*\varphi^*(D_t\otimes \pi^*M^{\vee})\\
&=\hat{\pi}_*(\varphi^*D_t\otimes\hat{\pi}^*\phi_{\hat{M}}^*M^{\vee})=\hat{\pi}_*(\varphi^*D_t\otimes\hat{\pi}^*(\oplus \hat{M}))=\oplus(\hat{\pi}_*\varphi^*D_t\otimes\hat{M})=\oplus(\mathcal{G}_t\otimes\hat{M}).
\end{align*}
Therefore, the index theorem of degree zero holds for the sheaf $\phi_{\hat{M}}^*(\mathcal{F}_t\otimes M^{\vee})$, i.e., for any $j\ge 1$ and $\hat{P}\in\textup{Pic}^0(\hat{A})=A$, 
we have the vanishing $H^j(\hat{A},\phi_{\hat{M}}^*(\mathcal{F}_t\otimes M^{\vee})\otimes\hat{P})=0$. 
Since $\phi_{\hat{M}}^*$ is finite, taking $\hat{P}:=\mathcal{O}_{\hat{A}}$, we see that $H^j(A,\mathcal{F}_t\otimes M^{\vee}\otimes (\phi_{\hat{M}})_*\mathcal{O}_{\hat{A}})=0$.
Now that $\mathcal{O}_A$ is a direct summand of $(\phi_{\hat{M}})_*\mathcal{O}_{\hat{A}}$ (cf.~\cite[Proposition 5.7]{KM98}),  our claim is thus  proved.
\end{proof}

We give a few remarks on {\hyperref[thm_alb]{Theorem   \ref*{thm_alb}}}.
First,  when $K_X+\Delta$ is pseudo-effective, \cite[Theorem 8.1]{LP20B} shows a stronger version of {\hyperref[thm_alb]{Theorem   \ref*{thm_alb}}} in higher dimensions. 
However, our {\hyperref[thm_alb]{Theorem   \ref*{thm_alb}}} here relaxes this assumption by assuming the nef divisor $L_X$ being strictly nef. 
Second, we note that, the first part of  {\hyperref[thm_alb]{Theorem   \ref*{thm_alb}}} actually holds for $n$-dimensional $X$ if we assume that  {\hyperref[main-conj-singular-arbitrary]{Question  \ref*{main-conj-singular-arbitrary}}} has a positive answer for the case when $\dim X\leqslant n-1$.

As a consequence of {\hyperref[thm_alb]{Theorem   \ref*{thm_alb}}}, we immediately have the corollary below. 
\begin{cor}[{cf.~{\hyperref[main_theorem_Goren_ter_3fold]{Theorem   \ref*{main_theorem_Goren_ter_3fold} (2)}}}]\label{coro_q0>0}
Let $X$ be a  normal projective threefold with at worst klt singularities, and $L_X$ a strictly nef $\mathbb{Q}$-divisor on $X$.
Suppose that the augmented irregularity $q^\circ(X)>0$.
Suppose further that either $X$ is $\mathbb{Q}$-factorial, or $X$ has only canonical singularities.
Then $K_X+tL_X$ is ample for sufficiently large $t$.
\end{cor}

\begin{proof}
Let $\pi:X'\to X$ be a quasi-\'etale cover such that $q(X')>0$, and $L_{X'}:=\pi^*L_X$.

First, we assume that $X$ has only canonical singularities. 
In this case, $X'$ also has only canonical singularities (cf.~\cite[Proposition 5.20]{KM98}).
By {\hyperref[thm_alb]{Theorem   \ref*{thm_alb}}}, our $K_{X'}+tL_{X'}$ and hence $K_X+tL_X$ are ample.

Second, we assume that $X$ is $\mathbb{Q}$-factorial.
In this case, applying \cite[Proposition 5.20]{KM98} again, our $X'$  has only klt singularities. 
Then, it follows from {\hyperref[thm_alb]{Theorem   \ref*{thm_alb}}} that our $K_{X'}+tL_{X'}\equiv E>0$ for some $t\gg 1$. 
So the push-down $\deg(\pi)\cdot(K_X+tL_X)$ is weakly numerically equivalent to $\pi_*(E)>0$ (cf.~\cite[Definition 2.2]{MZ18}). 
Since $X$ is $\mathbb{Q}$-factorial, weak numerical equivalence of divisors coincide with numerical equivalence (cf.~\cite[Lemma 3.2]{Zha16}); hence $K_X+tL_X\equiv \frac{1}{\deg(\pi)}\pi_*(E)$. 
By {\hyperref[prop-q-effective]{Proposition  \ref*{prop-q-effective}}}, our corollary follows.
\end{proof}

\section{Weak Calabi-Yau threefolds, Proof of \texorpdfstring{{\hyperref[main_theorem_Goren_ter_3fold]{Theorem  \ref*{main_theorem_Goren_ter_3fold} (3)}}}{text}}\label{sec-WCY}
The whole section is devoted to the proof of  {\hyperref[main_theorem_Goren_ter_3fold]{Theorem  \ref*{main_theorem_Goren_ter_3fold} (3)}}.  
To be more precise, we show that:
\begin{thm}[{cf.~{\hyperref[main_theorem_Goren_ter_3fold]{Theorem   \ref*{main_theorem_Goren_ter_3fold} (3)}}, \cite[Theorem 4.4]{Ser95}}]\label{thm-singular-cy}
Let $X$ be a normal projective variety with at worst terminal singularities, and $L_X$ a strictly nef divisor on $X$. 
If $\kappa(X)=0$ and $L_X\cdot c_2(X)\neq 0$, then $K_X+tL_X$ is ample for $t\gg 1$.
\end{thm}
Note that, for a normal projective variety which is of dimension $n$ and smooth in codimension two (this is the case when $X$ has only terminal singularities),  the second Chern class $\hat{c}_2(X)$ of the sheaf $\hat{\Omega}_X^1$ coincides with the birational second Chern class $c_2(X)$ (cf.~\cite[Proposition 1.1]{SW94} or \cite[Proof of Claim 4.11]{GKP16b}).
Here, the birational second Chern class $c_2(X)$ is defined to be the push-forward of the usual second Chern class $c_2(Y)$ of its smooth model $Y\to X$ (as $(n-2)$-cycles)  which is minimal in codimension two. 

In our present case $\kappa(X)=0$,  {\hyperref[main-conj-singular-arbitrary]{Question  \ref*{main-conj-singular-arbitrary}}} predicts that strict nefness is equivalent to ampleness, which relates to the abundance conjecture. 
We refer readers to \cite{LS20} for a recent progress on smooth threefolds in this direction.

To prove {\hyperref[thm-singular-cy]{Theorem  \ref*{thm-singular-cy}}}, we prepare the following lemma, as a natural extension of \cite[Proposition 3.7]{Ser95}. 
\begin{lemme}[{cf.~\cite[Proposition 3.7]{Ser95}}]\label{lem-Ser-unitary}
Let $X$ be a normal projective variety with at worst terminal singularities, and $L_X$ a strictly nef Cartier divisor on $X$.
Suppose that $\Coh^1(X,-2L_X)\neq 0$.
Then there exists a finite dimensional linear representation of the fundamental group   $\pi_1(X_{\textup{reg}})$ (where $X_{\textup{reg}}$ denotes the regular locus of $X$) whose image is not  abelian.
\end{lemme}

\begin{proof}
If $\dim X\le 2$, then $X$ is smooth and hence our lemma follows from \cite[Proposition 3.7]{Ser95}. 
Therefore, we may assume in the following that  $n=\dim X\ge 3$. 
We shall reduce the proof to the smooth surface, and then our lemma can be deduced from \cite[Proposition 3.7]{Ser95}. 

Let $A$ be a very ample divisor on $X$ and take $H_1,\cdots,H_{n-2}$ to be general hypersurfaces in the linear system $|A|$.
Let $S=H_1\cap \cdots\cap H_{n-2}$ be the complete intersection surface cut out by $H_1,\cdots, H_{n-2}$. 
Terminal varieties being smooth in codimension two (cf.~\cite[Corollary 5.18]{KM98}), our $S$ is smooth. 
Since $X_{\textup{reg}}$ has the same homotopy type of the space obtained from $H_1\cap X_{\textup{reg}}$ by attaching cells of dimension $\ge \dim X$ and the fundamental group of a CW complex only depends on its 2-skeleton, we have $\pi_1(X_{\textup{reg}})\cong\pi_1(H_1\cap X_{\textup{reg}})$.
Iterating this argument and noting that $S\cap X_{\textup{reg}}=S$, we see that $\pi_1(S)\cong \pi_1(X_{\textup{reg}})$. 

In addition, we have the following natural exact sequence
$$\Coh^1(X,-2L_X-H_1)\to \Coh^1(X,-2L_X)\to \Coh^1(H_1,-2L_X|_{H_1})\to\Coh^2(X,-2L_X-H_1).$$
By the Kodaira vanishing theorem, we have $\Coh^1(X,-2L_X-H_1)=\Coh^2(X,-2L_X-H_1)=0$ noting that $\dim X\ge 3$.
Hence, $\Coh^1(H_1,-2L_X|_{H_1})\cong\Coh^1(X,-2L_X)\neq 0$ by our assumption.
Inductively, we see that $\Coh^1(S,-2L_X|_S)\neq 0$. 
So we reduce the proof to the smooth surface case, and our lemma then follows from \cite[Proposition 3.7]{Ser95}.
\end{proof}

Now we begin to prove our main theorem of this section.
\begin{proof}[Proof of {\hyperref[thm-singular-cy]{Theorem  \ref*{thm-singular-cy}}} (and   {\hyperref[main_theorem_Goren_ter_3fold]{Theorem  \ref*{main_theorem_Goren_ter_3fold} (3)}})] In view of  {\hyperref[lem-big-ample]{Lemma   \ref*{lem-big-ample}}}, we
 may assume by contradiction that  $L_X$ is strictly nef but not big (and hence $L_X^3=0$). 
Applying {\hyperref[prop-klt-eff-ample]{Proposition  \ref*{prop-klt-eff-ample}}} and  {\hyperref[main_theorem_Goren_ter_3fold]{Theorem   \ref*{main_theorem_Goren_ter_3fold} (2)}} (cf.~{\hyperref[rmk-pf-mainthm]{Remark   \ref*{rmk-pf-mainthm} (2)}}),  
one can  reduce the proof to the case when $K_X\sim_\mathbb{Q} 0$ and $q^\circ(X)=0$, i.e., $X$ is a weak Calabi-Yau terminal threefold.  
Similarly, we can further assume that $L_X$ is not  numerically equivalent to any effective divisor (cf.~{\hyperref[prop-klt-eff-ample]{Proposition  \ref*{prop-klt-eff-ample}}}).  
Then it is clear that the Euler characteristic $\chi(\mathcal{O}_X)=\chi(K_X)=0$ by applying the Riemann-Roch formula and the Serre duality. 
By Miyaoka's pseudo-effectivity of $c_2(X)$ (cf.~\cite{Miy87}), we see that $L_X\cdot c_2(X)\ge 0$.
This together with our assumption implies $L_X\cdot c_2(X)>0$.
Since $K_X\sim_\mathbb{Q} 0$, after passing $X$ to its global index one cover, we may further assume that $K_X\sim 0$. 
Let $\gamma:\widetilde{X}\to X$ be a Galois, maximally quasi-\'etale cover, as constructed in \cite[Theorem 1.5]{GKP16b}; in other words, any finite \'etale cover of the regular locus $\widetilde{X}_{\textup{reg}}$ extends to a finite \'etale cover of $\widetilde{X}$ (as one of the equivalent conditions in \cite[Theorem 1.5]{GKP16b}). 
Since $\dim \widetilde{X}=3$ and $q^\circ(\widetilde{X})=q^\circ(X)=0$, it follows from \cite[Corollary 1.8]{Graf18} that $\pi_1(X)$ is finite (and hence virtually abelian). 
Therefore, replacing $\widetilde{X}$ by a further \'etale Galois cover, we may assume that $\pi_1(\widetilde{X})$ is abelian with  $\chi(\mathcal{O}_{\widetilde{X}})=q^\circ(\widetilde{X})=0$, noting that any \'etale cover of $\widetilde{X}$ also satisfies the conclusion of \cite[Theorem 1.5]{GKP16b}.

\textbf{We claim that any finite dimensional linear representation of $\pi_1(\widetilde{X}_{\textup{reg}})$ has abelian image.} Indeed, let $\rho:\pi_1(\widetilde{X}_{\textup{reg}})\to \textup{GL}_r(\mathbb{C})$ be such a finite dimensional linear representation of $\pi_1(\widetilde{X}_{\textup{reg}})$. 
Then it follows from \cite[Section 8.1, Proof of Theorem 1.14 and (8.0.5)]{GKP16b} that there exists a representation $\rho':\pi_1(\widetilde{X})\to \textup{GL}_r(\mathbb{C})$ such that $\rho'\circ\iota_*=\rho$ where $\iota_*:\pi_1(\widetilde{X}_{\textup{reg}})\twoheadrightarrow\pi_1(\widetilde{X})$ is induced by the natural inclusion (cf.~\cite[\S 0.7 (B)]{FL81}), noting that $\widetilde{X}$ is normal.
Therefore, $\textup{im}\rho=\textup{im}\rho'=\rho'(\pi_1(\widetilde{X}))$ is abelian and our claim is thus proved.

Now we shall finish the proof of our theorem by deducing a contradiction. 
Let $L_{\widetilde{X}}:=\gamma^*L_X$  and denote by $c_2(\widetilde{X})$ the birational second Chern class of $\widetilde{X}$, which coincides with $\gamma^*c_2(X)$ (cf. \cite[Proposition 4.6]{GKP16b}).   
By the construction,  $\chi(\mathcal{O}_{\widetilde{X}})=q^\circ(\widetilde{X})=0$,  $c_2(\widetilde{X})\cdot L_{\widetilde{X}}>0$,  $K_{\widetilde{X}}\sim 0$ and $\widetilde{X}$ has at worst terminal singularities (cf.~\cite[Proposition 5.20]{KM98}). 
Then it follows from the Serre duality that  $\textup{h}^2(\widetilde{X},\mathcal{O}_{\widetilde{X}})=0$ and $\textup{h}^3(\widetilde{X},\mathcal{O}_{\widetilde{X}})=1$.
On the one hand, by the Riemann-Roch formula, 
$$\chi(\widetilde{X},-2L_{\widetilde{X}})=-\frac{1}{6}L_{\widetilde{X}}\cdot c_2(\widetilde{X})<0,$$
noting that $L_{\widetilde{X}}^3=0$ since $L_{\widetilde{X}}$ is not big as we assumed in the beginning. 
On the other hand, as a consequence of  {\hyperref[prop-klt-eff-ample]{Proposition  \ref*{prop-klt-eff-ample}}} and Serre duality, we have 
$\Coh^3(\widetilde{X},-2L_{\widetilde{X}})=\Coh^0(\widetilde{X},2L_{\widetilde{X}})=0$, for otherwise, $L_{\widetilde{X}}$ would be numerically equivalent to an effective divisor and hence $L_{\widetilde{X}}$ is ample. 
Therefore, $\Coh^1(\widetilde{X},-2L_{\widetilde{X}})\neq 0$ and then 
 it follows from {\hyperref[lem-Ser-unitary]{Lemma   \ref*{lem-Ser-unitary}}} that there exists a linear representation of $\pi_1(\widetilde{X}_{\textup{reg}})$ whose image is not abelian, a contradiction to our claim.
\end{proof}

We end up this section with the following question on the weak Calabi-Yau fourfold case.
We note that, when $X$ is smooth, this question has a positive answer as shown in a recent joint paper \cite[Theorem 1.2]{LM21}  by Liu and Matsumura.

\begin{ques}
Let $X$ be a weak Calabi-Yau fourfold (in the sense that $K_X\sim_\mathbb{Q} 0$ and $q^\circ(X)=0$), and $L_X$ a strictly nef $\mathbb{Q}$-Cartier divisor on $X$.
Is $L_X$ ample?
\end{ques}

\section{The first step of the MMP}
In this section, 
we shall run the first step of the MMP for the proof of {\hyperref[main_theorem_Goren_ter_3fold]{Theorem   \ref*{main_theorem_Goren_ter_3fold} (4)}}  (cf.~{\hyperref[remark_conclude_excep_tocurve]{Remark   \ref*{remark_conclude_excep_tocurve}}}  and {\hyperref[thm_3fold_surface_curve]{Theorem    \ref*{thm_3fold_surface_curve}}}).

We begin this section with the lemma below, as  a special  case of \cite[Theorem 2.2]{Del14} which 
extends  \cite[Theorem 4]{Cut88} to the  isolated canonical singularities so as to avoid  small contractions.

\begin{lemme}[{cf.~\cite[Theorem 2.2]{Del14}}]\label{lem_del14}
Let $X$ be a $\mathbb{Q}$-factorial Gorenstein normal projective threefold with only isolated  canonical singularities. 
Let $\varphi:X\to Y$ be a birational $K_X$-negative contraction  of an extremal face $R$ whose fibres are at most one-dimensional.
Then the following assertions hold.
\begin{enumerate}
\item[(1)] 
The exceptional locus $E:=\textup{Exc}(\varphi)$ is a disjoint union of prime divisors, 
$\varphi$ is a composition of divisorial contractions mapping exceptional divisors onto curves, and $Y$ has only isolated canonical singularities.
\item[(2)] There exists a finite subset $T\subseteq Y$  such that $Y\backslash T\subseteq Y_{\textup{reg}}$,  $\text{codim}~\varphi^{-1}(T)\ge 2$, $X\backslash \varphi^{-1}(T)\subseteq X_{\textup{reg}}$ and 
$$\varphi|_{X\backslash\varphi^{-1}(T)}:X\backslash \varphi^{-1}(T)\to Y\backslash T$$
is the simultaneous blow-up of   smooth curves.
In particular, $K_X=\varphi^*(K_Y)+\sum E_i$ with $E_i$ being exceptional, and $\textup{Sing}~(\sum E_i)\subseteq \varphi^{-1}(\textup{Sing}~\sum\varphi(E_i))$.
\item[(3)] Let $f\subseteq X$ be an irreducible curve such that $[f]\in R$. Then
$K_X\cdot f=E\cdot f=-1$.
\end{enumerate}
\end{lemme}

Based on {\hyperref[lem_del14]{Lemma  \ref*{lem_del14}}}, we establish the following lemma on an elementary contraction (of an extremal ray), which is a key to reduce the isolated canonical singularities to terminal singularities via the terminalization.

\begin{lemme}\label{lem_canonical_terminal}
Let $X$ be a normal projective threefold with only isolated $\mathbb{Q}$-factorial Gorenstein canonical singularities. 
Suppose that $\pi:X\to Y$ is a birational elementary  contraction of a $K_X$-negative extremal ray.
Then we have the following commutative diagram
\[\xymatrix{\widetilde{X}\ar[r]^{\widetilde{\pi}}\ar[d]_{\tau_1}&\widetilde{Y}\ar[d]^{\tau_2}\\
X\ar[r]_{\pi}&Y
}
\]
such that the following assertions hold.
\begin{enumerate}
\item[(1)] $\tau_1$ is a (crepant) terminalization and $\widetilde{X}$ has only $\mathbb{Q}$-factorial Gorenstein terminal singularities.
\item[(2)] $\widetilde{\pi}$ is an elementary  contraction of a $K_{\widetilde{X}}$-negative extremal ray, which is   divisorial.
\item[(3)] $Y$ has at worst $\mathbb{Q}$-factorial isolated canonical singularities.
\end{enumerate}
Suppose further that either $\pi$ maps the exceptional divisor $E:=\textup{Exc}(\pi)$ to a curve, or $\widetilde{\pi}$ maps the exceptional divisor $\widetilde{E}:=\textup{Exc}(\widetilde{\pi})$ to a single point.
Then we further have:
\begin{enumerate}
\item[(4)] $\widetilde{E}\cap\textup{Exc}(\tau_1)=\emptyset$ and $\widetilde{E}=\tau_1^{-1}(E)\cong E$;
\item[(5)]  $\tau_2$ is also crepant; and in particular, 
\item[(6)] if $\pi(E)$ is a curve, then $X$ is the blow-up of $Y$ along a local complete intersection curve $C$, $Y$ is  Gorenstein which is smooth around $C$, and $E$ is  Cartier; moreover, $\pi_*(-E|_E)=\pi(E)=C$.
\end{enumerate}
\end{lemme}

\begin{proof}
Let $\tau_1:\widetilde{X}\to X$ be a $\QQ$-factorial terminalization (cf. \cite[Corollary 1.4.3]{BCHM10}, by which $\tau_1$ is crepant), then $\widetilde{X}$ has only $\mathbb{Q}$-factorial Gorenstein terminal singularities.
Let $\widetilde{C}\subseteq\widetilde{X}$ be an irreducible curve such that $\tau_1(\widetilde{C})$ is $\pi$-contracted.
Then $K_{\widetilde{X}}\cdot \widetilde{C}=K_X\cdot\tau_1(\widetilde{C})<0$.
Therefore, there is an extremal ray $\widetilde{R}\in\overline{\textup{NE}}(\pi\circ\tau_1)\subseteq\overline{\textup{NE}}(\widetilde{X})$ such that $K_{\widetilde{X}}\cdot\widetilde{R}<0$.
Let $\widetilde{\pi}:\widetilde{X}\to \widetilde{Y}$ be the  contraction of $\widetilde{R}$.
By the rigidity lemma (cf.~\cite[Lemma 1.15, pp.~12-13]{Deb01}) $\pi\circ\tau_1$ factors through $\widetilde{\pi}$.
So $\widetilde{\pi}$ is birational and thus divisorial (see \cite[Theorem 4]{Kaw84} and \cite[Theorem 0]{Ben85}; cf.~\cite[Introduction]{Cut88}). 
Hence, (1) and (2) are proved. 
By {\hyperref[lem_del14]{Lemma  \ref*{lem_del14}}}, $\pi$ is  divisorial; hence (3) follows from \cite[Proposition 3.36]{KM98}, noting that non-terminal points of $Y$  are contained in the images of the non-terminal points of $X$ (cf. \cite[Corollary 3.43, pp.~104-105]{KM98}).
Let $\widetilde{F}$ be any curve contracted by $\widetilde{\pi}$ and set $F:=\tau_1(\widetilde{F})$.
Then $K_X\cdot F=K_{\widetilde{X}}\cdot\widetilde{F}<0$ implies that $F$ is still a curve and contracted by $\pi$.
Hence, $\tau_1(\widetilde{E})\subseteq E$, but $E$ being irreducible (since $\pi$ is elementary), this implies that $\tau_1(\widetilde{E})=E$.

From now on, we assume that either $\pi(E)$ is a curve or $\widetilde{\pi}(\widetilde{E})$ is a single point.
\textbf{We claim that $E$ does not contain any non-terminal point of $X$.}
Suppose the contrary.
Since $X$ is $\mathbb{Q}$-factorial, it follows from \cite[Corollary 2.63]{KM98} that $\textup{Exc}(\tau_1)$ is of pure codimension one.
Therefore, $\tau_1^*(E)=\widetilde{E}+F$ with $F$ being $\tau_1$-exceptional, noting that $\widetilde{E}$ is an irreducible divisor since $\widetilde\pi$ is elementary by (2).
Since $\widetilde{X}$ is also $\mathbb{Q}$-factorial, we take a curve $B$ in $\widetilde{E}\cap F$. Then from $K_{\widetilde{X}}\cdot B=\tau_1^*K_X\cdot B=0$ we see that $B$ is not contracted by $\widetilde\pi$, hence $\widetilde\pi(\widetilde B)$ is a curve and since $\widetilde E$ is irreducible we must have $\widetilde\pi(\widetilde E)=\widetilde\pi(B)$. 
If $\widetilde{E}$ is mapped onto a single point, then so is $B$, a contradiction.
If $E$ is mapped onto a curve, 
then $\widetilde{\pi}(\widetilde{E})=\widetilde{\pi}(B)$ is  contracted by $\tau_2$, implying that $\tau_2(\widetilde{\pi}(\widetilde{E}))$ is  a single point on $Y$, contradicting  the assumption that $\pi(E)$ is a curve.
So our claim holds and (4) is proved.
(5) follows from the the diagram (for some non-negative number $a$):
$$K_{\widetilde{Y}}=\widetilde{\pi}_*(K_{\widetilde{X}}-a\widetilde{E})=\widetilde{\pi}_*\tau_1^*\pi^*K_Y=\tau_2^*K_Y.$$
Now, we show (6). 
If $C:=\pi(E)$ is a curve, so is $\widetilde{\pi}(\widetilde{E})$; hence $K_{\widetilde{Y}}$ is Cartier
(cf.~\cite[Lemma 3]{Cut88}).
Since $\widetilde{E}$ is disjoint with $\textup{Exc}(\tau_1)$, we see that $\widetilde{X}$ and $X$ are isomorphic around $\widetilde{E}$, and hence $X$ is terminal around $E$.
Following from \cite[Corollary 3.43, pp.~104-105]{KM98}, $Y$ is also terminal around the image $C=\pi(E)$.
In addition, the restrictions $\widetilde{\pi}|_{\widetilde{E}}$ and $\pi|_E$ are the same contraction; therefore, $\widetilde{Y}$ and $Y$ are isomorphic around $\widetilde{C}$. 
By \cite[Theorem 4]{Cut88},  $X$ is the blow-up of a local complete intersection curve on $Y$ and $Y$ is smooth around $C$. 
Since $K_{\widetilde{Y}}=\tau_2^*K_Y$, our $K_Y$ is also Cartier: indeed, since $\tau_2$ is birational, the function field $K(\widetilde{Y})$  coincides with $K(Y)$; hence the local defining  equation for $K_{\widetilde{Y}}$ on $\widetilde{Y}$ is also that for $K_Y$. 
Applying {\hyperref[lem_del14]{Lemma  \ref*{lem_del14} (2)}} to $\pi$, we see that $E$ is also Cartier. 
Then  
 $\mathcal{O}_X(-E)|_E\cong\mathcal{O}_E(1)$. 
Since $-E|_E$ is $\pi|_E$-ample, we have   $q(-E|_E)+G\sim H_E$ for some $q>0$, some fibre $G$ of $\pi|_E:E\to C$, and some very ample curve $H_E$ on $E$.
By the intersection theory and {\hyperref[lem_del14]{Lemma  \ref*{lem_del14}}}, $\deg (H_E\to C)=q$, so $(6)$ is proved.
\end{proof}

\begin{rmq}
If $\pi$ contracts $E$ to a point in {\hyperref[lem_canonical_terminal]{Lemma  \ref*{lem_canonical_terminal}}}, then  {\hyperref[lem_canonical_terminal]{Lemma  \ref*{lem_canonical_terminal} (4)}} may not be true, since $\widetilde{\pi}$ in this case will  possibly contract $\widetilde{E}$ to a curve, and then $\widetilde{E}$ is the blow-up of $E$ along some point; in other words, $E$ may contain some non-terminal point of $X$.
\end{rmq}

The following 
{\hyperref[contr_curve_elementary]{Proposition   \ref*{contr_curve_elementary}}} $\sim$ {\hyperref[prop_contr_surface_elementary]{Proposition   \ref*{prop_contr_surface_elementary}}} describe the first step of the MMP which is either a Fano contraction or a divisorial contraction mapping the exceptional divisor to a single point (cf.~{\hyperref[remark_conclude_excep_tocurve]{Remark  \ref*{remark_conclude_excep_tocurve}}}).
Considering {\hyperref[main-conj-singular-arbitrary]{Question    \ref*{main-conj-singular-arbitrary}}}, we shall make our arguments as general as possible.

\begin{prop}\label{contr_curve_elementary}
Let $(X,\Delta)$ be a projective klt threefold pair, and $L_X$ a strictly nef divisor on $X$.
Suppose  $\rho(X)\leqslant 2$ (this is the case when $X$ admits a $(K_X+\Delta)$-negative  elementary contraction to a curve).
Then $K_X+\Delta+tL_X$ is ample for  $t\gg 1$.
\end{prop}
\begin{proof}
If view of  \cite[Theorem D]{LOWYZ21}, we may assume $\rho(X)=2$ and $K_X+\Delta$ is not parallel to $L_X$.  
Hence, the Mori cone $\overline{\textup{NE}}(X)$ has only two extremal rays. 
In particular, some linear combination $a(K_X+\Delta)+bL_X$ is strictly positive on $\overline{\textup{NE}}(X)\backslash \{0\}$ and thus ample (cf.~\cite[Theorem 1.18]{KM98}).
By {\hyperref[lem-big-ample]{Lemma  \ref*{lem-big-ample}}}, our proposition is proved.
\end{proof}

\begin{prop}\label{prop_Exc_surface_point}
Let $(X,\Delta)$ be a $\mathbb{Q}$-factorial normal  projective threefold with only  klt singularities and $L_X$ a strictly nef divisor.
Suppose there is a divisorial contraction $\pi:X\to X'$ of a $(K_X+\Delta)$-negative extremal ray, mapping the exceptional divisor to a point.	 
Then $K_X+\Delta+tL_X$ is ample for  $t\gg1$.
\end{prop}

\begin{proof}
In view of {\hyperref[lem-big-ample]{Lemma  \ref*{lem-big-ample}}}, we may assume that $K_X+\Delta+tL_X$ is not big.
By {\hyperref[lem-not-big-some]{Lemma  \ref*{lem-not-big-some}}}, $(K_X+\Delta)^i\cdot L_X^{3-i}=0$ for any $0\leqslant i\leqslant 3$.
Write $(K_X+\Delta)-aE=\pi^*(K_{X'}+\Delta')$ and $L_X+bE=\pi^*L_{X'}$ with $L_{X'}$ being  $\mathbb{Q}$-Cartier  (cf.~\cite[Theorem 3-2-1]{KMM87}) and $E$ being  $\pi$-exceptional. 
Since $-E$ is $\pi$-ample (cf.~\cite[Lemma 2.62]{KM98}) and $L_X$ is strictly nef, our $b>0$. 
By the projection formula, $L_{X'}$ is strictly nef.
Note that 
$\pi^*(K_{X'}+\Delta'+tL_{X'})=(K_X+\Delta+tL_X)+(bt-a)E$;  
hence for  $t\gg 1$, our $K_{X'}+\Delta'+tL_{X'}$ is strictly nef.

We claim that $K_{X'}+\Delta'+tL_{X'}$ is big (and hence ample by {\hyperref[lem-big-ample]{Lemma  \ref*{lem-big-ample}}}). 
Suppose the claim for the time being.
Fixing some $m\gg 1$, we choose a  smooth member $A\in |m(K_{X'}+\Delta'+tL_{X'})|$ which is ample.
Let $D_X:=b(K_X+\Delta)+aL_X$ and $D_{X'}:=b(K_{X'}+\Delta')+aL_{X'}$.
Then  $D_X=\pi^*(D_{X'})$.
Since $E$ is mapped to a point, $D_X^2\cdot(K_X+\Delta+tL_X)=0$ implies $D_{X'}^2\cdot(K_{X'}+\Delta'+tL_{X'})=0$.
Similarly, $D_{X'}\cdot (K_{X'}+\Delta'+tL_{X'})^2=D_X\cdot (K_X+\Delta+tL_X+(bt-a)E)^2=0$.  
So $D_{X'}^2\cdot A=D_{X'}\cdot  A^2=0$.
By \cite[Lemma 3.2]{Zha16}, we have $D_{X'}\equiv 0$. 
Then, $D_X\equiv0$ and   $-(K_X+\Delta)$ is parallel to $L_X$, which is strictly nef.
As a result, $-(K_X+\Delta)$ is ample (cf.~\cite[Theorem D]{LOWYZ21}) and our result follows.

It remains to show the bigness of $K_{X'}+\Delta'+tL_{X'}$.
Set $M_X:=K_X+\Delta+tL_X$, $M_{X'}:=K_{X'}+\Delta'+tL_{X'}$ and $m:=bt-a$.
Then $\pi^*M_{X'}=M_X+mE$. 
Since $M_{X'}$ is nef and $\pi(E)$ is  a single point, by projection formula we have
\begin{align*}
	0\leqslant M_{X'}^3&=\pi^*M_{X'}^2\cdot (M_X+mE)=\pi^*M_{X'}\cdot  (M_X+mE)\cdot M_X\\
	&=\pi^*M_{X'}\cdot M_X^2+\pi^*M_{X'}\cdot mM_X|_E\\
	&=\pi^*M_{X'}\cdot M_X^2=(bt-a)(K_X+\Delta+tL_X)^2\cdot E
\end{align*}
In the above equality, $-E$ is $\pi$-ample and we regard $M_X|_E$ as a 1-cycle on $E$ (which is also $\pi$-contracted). 
Thus, $(L_X|_E)^2=L_X^2\cdot E=(\pi^*L_{X'}-bE)^2\cdot E=(-bE|_E)^2>0$.
By the  Nakai-Moishezon criterion (cf.~\cite[Theorem 1.37]{KM98}),  $L_X|_E$ is ample.
So the last item is positive for  $t\gg 1$.
In particular, $M_{X'}=K_{X'}+\Delta'+tL_{X'}$ is big (cf.~\cite[Proposition 2.61]{KM98}).
\end{proof}

\begin{prop}
\label{prop_contr_surface_elementary}
Let $X$ be a $\mathbb{Q}$-factorial normal projective threefold with only isolated  klt singularities,  and $L_X$ a strictly nef $\mathbb{Q}$-divisor on $X$.
Suppose  there is an elementary  $K_X$-negative contraction  $\pi:X\to S$ onto a normal projective surface $S$.
Then $K_X+tL_X$ is ample for $t\gg 1$.
\end{prop}

\begin{proof}
Let us assume that $L_X$ is Cartier after replacing $L_X$ by a multiple. 
By {\hyperref[coro_q0>0]{Corollary   \ref*{coro_q0>0}}}, {\hyperref[lem_strict_nef_k]{Lemma  \ref*{lem_strict_nef_k}}} and {\hyperref[lem-big-ample]{Lemma  \ref*{lem-big-ample}}}, we may assume $q^\circ(X)=0$, and $K_X+tL_X$ is nef but not big for  $t\gg 1$. 
Note that we can further assume $K_X$ not parallel to $L_X$ (cf.~\cite[Theorem D]{LOWYZ21}).  
By the canonical bundle formula (cf.~e.g. \cite[Theorem 0.2]{Amb05}), $(S,\Delta_S)$ is a klt pair for some effective divisor $\Delta_S$ on $S$; thus $S$ is $\mathbb{Q}$-factorial (cf.~e.g.~\cite[Proposition 4.11]{KM98}).
Let $u:=\frac{-K_X\cdot f}{L_X\cdot f}>0$ (which is a rational number) where $f$ is a general fibre of $\pi$.
By the cone theorem (cf.~\cite[Theorem 3.7]{KM98}), there exists a $\mathbb{Q}$-Cartier divisor $M$ such that
$K_X+uL_X=\pi^*M$ with $M\not\equiv 0$.  
Then, {\hyperref[lem-not-big-some]{Lemma  \ref*{lem-not-big-some}}} gives us 
\begin{align*}
0=(K_X-\pi^*M)^3=-3K_X^2\cdot(K_X+uL_X)+3K_X\cdot(\pi^*M)^2=3(K_X\cdot f)M^2,
\end{align*}
which implies  $M^2=0$.  
Fix a rational number $t>6$ such that  $K_X+tL_X$ is strictly nef.  
By {\hyperref[lem-not-big-some]{Lemma  \ref*{lem-not-big-some}}},  there exists  $\alpha\in\overline{\textup{ME}}(X)$ such that $K_X\cdot\alpha=L_X\cdot \alpha=0$.
Then $M\cdot \gamma=0$ with $0\not\equiv\gamma:=\pi_*\alpha\in\overline{\textup{ME}}(S)$.
Since $\dim S=2$, our $\gamma$ is nef.
Since $M^2=0$ but $M\not\equiv0$, we have $\gamma^2=0$ and thus $\gamma$ is parallel to $M$ by the Hodge index theorem. 
Hence, either $M$  is nef or $-M$ is nef.
Denote by $M_1=M$ (resp. $-M$) if $M$ is nef (resp. $-M$ is nef). 
Replacing $M_1$ by a multiple, we may assume that $M_1$ is a line bundle.

\begin{claim}\label{claim_m-k-pseu}
$N:=M-K_S$	is pseudo-effective.
\end{claim}

\noindent
\textbf{Proof of  {\hyperref[claim_m-k-pseu]{Claim \ref*{claim_m-k-pseu}}}.}
For any irreducible curve $\ell\subseteq S$, it follows from {\hyperref[lem_conic_miy]{Lemma    \ref*{lem_conic_miy}}} that
\begin{align*}
 0\leqslant  u^2L_X^2\cdot\pi^*(\ell)&=(\pi^*M-K_X)^2\cdot\pi^*(\ell)=-2\pi^*M\cdot K_X\cdot \pi^*(\ell)	+K_X^2\cdot\pi^*(\ell)\\
 &=-2(K_X\cdot f)M\cdot \ell-(4K_S+D_1)\cdot \ell=(4N-D_1)\cdot \ell.
\end{align*}
Here,  $D_1$ denotes the one-dimensional part of the discriminant locus of  $\pi$.  
Then $4N-D_1$ is nef and hence
$N=\frac{1}{4}((4N-D_1)+D_1)$ is pseudo-effective.

\par \vskip 1pc \noindent

We come back to the proof of {\hyperref[prop_contr_surface_elementary]{Proposition   \ref*{prop_contr_surface_elementary}}}. 
Let $\tau:\widetilde{S}\to S$ be a minimal resolution, and  $\widetilde{M_1}:=\tau^*M_1$. 
Since  $M_1\cdot M=0$ and $M_1\cdot N\geqslant 0$ (cf.~{\hyperref[claim_m-k-pseu]{Claim \ref*{claim_m-k-pseu}}}), our $K_{\widetilde{S}}\cdot \widetilde{M_1}=K_S\cdot M_1\leqslant 0$. 
Then, for any positive integer $n$, applying the Riemann-Roch formula to $n\widetilde{M_1}$ and noting that $q(\widetilde{S})=q(S)=q(X)=0$, we get the following inequality
\begin{align}\label{eq_RR}
\begin{split}
h^0(S,nM_1)&=h^0(\widetilde{S},n\widetilde{M_1})=-\frac{n}{2}\widetilde{M_1}\cdot K_{\widetilde{S}}+\chi(\mathcal{O}_{\widetilde{S}})+h^1(\widetilde{S},n\widetilde{M_1})-h^2(\widetilde{S},n\widetilde{M_1})\\
&\geqslant 1-h^2(\widetilde{S},n\widetilde{M_1}).
\end{split}
\end{align}

\textbf{We claim that $h^2(\widetilde{S},n\widetilde{M_1})=h^0(K_{\widetilde{S}}-n\widetilde{M_1})=0$ for $n\gg 1$.}  
Fixing an ample divisor $\widetilde{H}$ on $\widetilde{S}$, we have $\widetilde{M_1}\cdot \widetilde{H}>0$ by the Hodge index theorem (noting that $\widetilde{M_1}^2=M_1^2=0$).
For each $n$, if there is an effective divisor $Q_n$  on $\widetilde{S}$ such that $K_{\widetilde{S}}-n\widetilde{M_1}\sim Q_n$, then 
$$K_{\widetilde{S}}\cdot \widetilde{H}=(n\widetilde{M_1}+Q_n)\cdot\widetilde{H}\geqslant  n\widetilde{M_1}\cdot\widetilde{H}.$$
The left hand side of the above inequality being bounded, our claim is thus proved.

Therefore, it follows from {\hyperref[eq_RR]{Equation  (\ref*{eq_RR})}}  that $M_1$ is numerically equivalent to an effective $\mathbb{Q}$-divisor, which is non-zero by our assumption in the beginning of the proof.
This implies that 
 $K_X+uL_X$ (or $-K_X-uL_X$) is  numerically equivalent to  a non-zero effective divisor. 
By {\hyperref[prop-q-effective]{Proposition  \ref*{prop-q-effective}}}, our proposition is proved.
\end{proof}

The following theorem generalizes {\hyperref[prop_contr_surface_elementary]{Proposition   \ref*{prop_contr_surface_elementary}}} to the non-elementary case.

\begin{thm}\label{thm_contr_surface_conic}
Let $X$ be a $\mathbb{Q}$-factorial Gorenstein normal projective threefold with only isolated  canonical singularities,  and $L_X$ a strictly nef divisor on $X$.
Suppose that $X$ admits an equi-dimensional  $K_X$-negative contraction (of an extremal face) $\pi:X\to S$ onto a normal projective surface such that $K_X+uL_X=\pi^*M$ for some $u\in\mathbb{Q}$ and some $\mathbb{Q}$-Cartier divisor $M$ on $S$.
Then $K_X+tL_X$ is ample for $t\gg 1$.
\end{thm}

\begin{rmq}\label{rem_composition_conic}
With the same assumption as in {\hyperref[thm_contr_surface_conic]{Theorem    \ref*{thm_contr_surface_conic}}},  applying \cite[Proof of Proposition 3.4]{Rom19} and {\hyperref[lem_del14]{Lemma  \ref*{lem_del14}}}, we 
get the following MMP for $X$ over $S$, noting that the Gorenstein condition on $X$ can descend to each $X_i$ (cf.~{\hyperref[lem_canonical_terminal]{Lemma  \ref*{lem_canonical_terminal} (6)}}).
\[\xymatrix{X=:X_0\ar[r]^{\phi_0}\ar[drrr]_{\pi_0:=\pi}&X_1\ar[r]^{\phi_1}&\cdots\ar[r]^{\phi_{n-1}}&X_n\ar[d]^{\pi_n}\\
&&&S
}
\]
Here, each $\phi_i$ is an elementary $K_{X_i}$-negative divisorial contraction mapping the exceptional divisor  onto a curve (cf.~{\hyperref[lem_del14]{Lemma  \ref*{lem_del14}}}), each $X_i$ is a normal projective threefold with only isolated $\mathbb{Q}$-factorial Gorenstein canonical singularities  (cf.~{\hyperref[lem_canonical_terminal]{Lemma  \ref*{lem_canonical_terminal}}}), and $\pi_n$ is an elementary conic fibration in the sense that the generic fibre is a smooth plane conic. 

With this in mind, we extend the formula in {\hyperref[lem_conic_miy]{Lemma  \ref*{lem_conic_miy}}} to the non-elementary case in our situation.
Let $\phi:=\phi_{n-1}\circ\cdots\circ\phi_0$.
By {\hyperref[lem_del14]{Lemma  \ref*{lem_del14} (2)}} and {\hyperref[lem_canonical_terminal]{Lemma  \ref*{lem_canonical_terminal} (6)}},  $K_X=\phi^*K_{X_n}+\sum E_i$ with $E_i$  being pairwise disjoint, and $\phi_*(-E_i|_{E_i})=\phi(E_i)$.
Then for any divisor $H$ on $X_n$, we have 
\begin{align*}
	\phi_*(K_X^2)\cdot H&=(\phi^*K_{X_n}\cdot\phi^*K_{X_n}+2\sum\phi^*K_{X_n}\cdot E_i+\sum E_i|_{E_i})\cdot \phi^*H\\
	&=K_{X_n}^2\cdot H-\sum \phi(E_i)\cdot H
\end{align*}
where $\phi(E_i)$ are pairwise disjoint  curves and all of them are not $\pi_n$-contracted. 
Combining the above equality with {\hyperref[lem_conic_miy]{Lemma    \ref*{lem_conic_miy}}}, we have the following numerical equivalence    
$\pi_*(K_X^2)\equiv (\pi_n)_*(K_{X_n}^2)-\sum a_i\pi(E_i)=:-(4K_S+D)$,
where each $a_i>0$, and $D$ is the sum of the one-dimensional part of the discriminant locus of $\pi_n$ and $\sum a_i\pi(E_i)$; hence, $D$ is  effective.
\end{rmq}

\begin{proof}[Proof of {\hyperref[thm_contr_surface_conic]{Theorem    \ref*{thm_contr_surface_conic}}}]
The  proof is completely the same as {\hyperref[prop_contr_surface_elementary]{Proposition   \ref*{prop_contr_surface_elementary}}} after replacing {\hyperref[lem_conic_miy]{Lemma    \ref*{lem_conic_miy}}} by {\hyperref[rem_composition_conic]{Remark  \ref*{rem_composition_conic}}}.
\end{proof}

\begin{rmq}\label{remark_conclude_excep_tocurve}
We shall finish the proof of  {\hyperref[main_theorem_Goren_ter_3fold]{Theorem   \ref*{main_theorem_Goren_ter_3fold} (4)}} in  {\hyperref[sec_uniruled]{Section    \ref*{sec_uniruled}}}.  
In view of {\hyperref[lem_del14]{Lemma  \ref*{lem_del14}}} and {\hyperref[contr_curve_elementary]{Proposition   \ref*{contr_curve_elementary}}}  $\sim$ {\hyperref[prop_contr_surface_elementary]{Proposition   \ref*{prop_contr_surface_elementary}}}  above,  we only need to verify the case when all the elementary $K_X$-negative  contractions (of extremal rays) are divisorial contractions, mapping the exceptional divisors to curves, 
which is {\hyperref[thm_3fold_surface_curve]{Theorem \ref*{thm_3fold_surface_curve}}}.
\end{rmq}

\section{The uniruled case, Proof of \texorpdfstring{{\hyperref[main_theorem_Goren_ter_3fold]{Theorem   \ref*{main_theorem_Goren_ter_3fold} (4)}}}{text}}\label{sec_uniruled}

Throughout this section, we prove {\hyperref[main_theorem_Goren_ter_3fold]{Theorem   \ref*{main_theorem_Goren_ter_3fold} (4)}}. 
More precisely, we prove 
{\hyperref[thm_3fold_surface_curve]{Theorem \ref*{thm_3fold_surface_curve}}} below.

\begin{thm}[{cf.~\cite[Proposition 5.2]{CCP08}}]\label{thm_3fold_surface_curve}
Let  $X$ be a  $\mathbb{Q}$-factorial Gorenstein normal projective uniruled threefold with only isolated  canonical singularities, 
and $L_X$ a strictly nef divisor.
Suppose  all the elementary $K_X$-negative extremal contractions  are divisorial, mapping the exceptional divisors to curves.
Then $K_X+tL_X$ is ample for $t\gg 1$.
\end{thm}

\begin{notation}\label{not_bir_surface_curve}
We shall always stick to the following notations 
 and apply the induction on the Picard number $\rho(X)$ of $X$. 
\begin{enumerate}
\item[(1)] By {\hyperref[main_theorem_Goren_ter_3fold]{Theorem   \ref*{main_theorem_Goren_ter_3fold} (2)}}, we may assume the augmented irregularity $q^\circ(X)=0$.
\item[(2)] Let $\varphi_i:X\to X_i$ be the contraction of the $K_X$-negative extremal ray $\mathbb{R}_{\geqslant 0}[\ell_i]$, with the exceptional divisor $E_i$ mapped to a (possibly singular) curve $C_i$ on $X_i$. 
Let $\ell_i\cong\mathbb{P}^1$ be the general fibre such that $K_X\cdot\ell_i=E_i\cdot \ell_i=-1$ (cf.~{\hyperref[lem_del14]{Lemma  \ref*{lem_del14}}}).
\item[(3)] By {\hyperref[lem_canonical_terminal]{Lemma  \ref*{lem_canonical_terminal}}}, each $X_i$ also has at worst isolated $\mathbb{Q}$-factorial Gorenstein canonical singularities  with $\rho(X_i)=\rho(X)-1$.
\item[(4)] Let $I$ be the index recording the elementary $K_X$-negative extremal contractions.
\item[(5)] Let $\nu:=\min\left\{\frac{L_X\cdot\ell_i}{-K_X\cdot\ell_i}~|~i\in I\subseteq\mathbb{N}\right\}=\min\{L_X\cdot\ell_i~|~i\in I\}$.
Then $D_X:=L_X+\nu K_X$ is nef.
In view of \cite[Theorem D]{LOWYZ21}, we may assume that $D_X\not\equiv 0$.
\item[(6)] Let $I_0\subseteq I$ be the subset such that $i\in I_0$ if and only if $L_X\cdot \ell_i=\nu$.
\item[(7)] Let $\varphi=\varphi_1:X\to X_1=:X'$, $E:=E_1$, $L_{X'}:=\varphi_*L_X$ and $C':=\varphi(E)$ with $1\in I_0$. 
\item[(8)] Write $K_X=\varphi^*K_{X'}+E$, and $L_X=\varphi^*L_{X'}-\nu E$ (cf.~{\hyperref[lem_del14]{Lemma  \ref*{lem_del14}}}).
Let $D_{X'}:=L_{X'}+\nu K_{X'}$.
Then $D_X=\varphi^*D_{X'}$ and thus $D_{X'}$ is also nef.
\end{enumerate}	
\end{notation}

\textbf{Now we begin to prove {\hyperref[thm_3fold_surface_curve]{Theorem \ref*{thm_3fold_surface_curve}}}. 
If $\rho(X)\leqslant 2$, then our theorem follows from {\hyperref[contr_curve_elementary]{Proposition   \ref*{contr_curve_elementary}}}.
Suppose that our theorem holds for the case $\rho(X)\leqslant p$.
We shall assume $\rho(X)=p+1$ in the following.}

\begin{lemme}\label{lem_L.C>0}
Either $K_X+tL_X$ is ample for $t\gg 1$, or $L_{X'}\cdot C'\leqslant 0$.
\end{lemme}
\begin{proof}
Suppose the contrary that $K_X+tL_X$ is not ample for any $t\gg 1$ and $L_{X'}\cdot C'>0$. 
Then $L_{X'}$ is strictly nef on $X'$ and $K_{X'}+tL_{X'}$ is ample for  $t\gg 1$ by the induction on $\rho(X')$ (cf.~{\hyperref[remark_conclude_excep_tocurve]{Remark  \ref*{remark_conclude_excep_tocurve}}}). 
\textbf{We claim that $D_{X'}$ is semi-ample.} 
Consider the following 
$$\frac{2}{\nu}D_{X'}-K_{X'}=2(K_{X'}+\frac{1}{\nu}L_{X'})-K_{X'}=(K_{X'}+\frac{3}{2\nu}L_{X'})+\frac{1}{2\nu}L_{X'}.$$
Since $K_{X'}+\frac{3}{2\nu}L_{X'}=\frac{1}{\nu}D_{X'}+\frac{1}{2\nu}L_{X'}$ with $D_{X'}$ and $L_{X'}$ being nef, it must be big, for otherwise, $D_{X'}^3=D_{X'}^2\cdot L_{X'}=D_{X'}\cdot L_{X'}^2=L_{X'}^3=0$ would imply  $(K_{X'}+tL_{X'})^3=0$, a contradiction.
Hence, $\frac{2}{\nu}D_{X'}-K_{X'}$ is nef and big.
By the base-point-free theorem (cf.~\cite[Theorem 3.3]{KM98}), 
$D_{X'}$ is semi-ample. 
By {\hyperref[not_bir_surface_curve]{Notation  \ref*{not_bir_surface_curve} (5)}}
 and {\hyperref[prop-q-effective]{Proposition  \ref*{prop-q-effective}}}, our $K_X+tL_X$ is ample for $t\gg 1$,  a contradiction.
\end{proof}

\begin{lemme}\label{lem_D.C>0}
Either $K_X+tL_X$ is ample for $t\gg 1$, or $D_{X'}\cdot C'>0$.	
\end{lemme}

\begin{proof}
Suppose the contrary that $K_X+tL_X$ is not ample for any $t\gg 1$, and $D_{X'}\cdot C'=0$. 
Then $K_{X'}\cdot C'\geqslant 0$ (cf.~{\hyperref[lem_L.C>0]{Lemma \ref*{lem_L.C>0}}}); hence $(K_X-E)|_E=\varphi^*K_{X'}|_E$ is a nef divisor on $E$ (cf.~{\hyperref[lem_del14]{Lemma  \ref*{lem_del14}}} and {\hyperref[not_bir_surface_curve]{Notation  \ref*{not_bir_surface_curve} (5)}}).
Besides, $D_{X'}\cdot C'=0$ implies that our $D_X|_E\equiv 0$;  hence $L_X|_E\equiv-\nu K_X|_E$. 
So our $-K_X|_E$ is strictly nef on $E$.  
Together with the nefness of $(K_X-E)|_E$,
our $-E|_E$ is also strictly nef on $E$. 
We consider the following  diagram.
\[\xymatrix@C=5em{
\widetilde{E}\ar[r]^\sigma\ar[d]_\tau\ar[dr]^q&E_N\ar[r]^{p_N}\ar[d]^{n_E}&C'_N\ar[d]^{n_{C'}}\\
\widetilde{E}_m\ar[d]&E\ar[r]^p\ar@{^(->}[d]&C'\ar@{^(->}[d]\\
C'_N&X\ar[r]^\varphi&X'
}
\]
Here, $p:=\varphi|_E$, $n_\bullet$ is the normalization, $\sigma$ is the minimal resolution, $\tau$ is the MMP of $\widetilde{E}$, $p_N$ is the induced map and $q:=n_E\circ\sigma$.
Then, $\widetilde{E}_m$ is a ruled surface. 
Denote by $C_0\subseteq \widetilde{E}_m$ the section with the minimal self-intersection $(C_0)^2=-e$, and
$f$ a general fibre of $\widetilde{E}_m\to C_N'$.

As is shown in  {\hyperref[prop_Q-Goren_surface]{Proposition  \ref*{prop_Q-Goren_surface}}}, $K_{\widetilde{E}}\sim_{\mathbb{Q}}q^*K_E-F$, where $F$ is an effective divisor with $\varphi(q(F))\subseteq \textup{Sing}\,(C')$ (cf.~\cite[Lemma 5-1-9]{KMM87}, \cite[(4.1)]{Sak84} and {\hyperref[lem_del14]{Lemma  \ref*{lem_del14}}}).
Let $\widetilde{C}\subseteq \widetilde{E}$ be the strict transform of $C_0$  and $C_X:=q(\widetilde{C})$.
Since $C_X\not\subseteq\textup{Sing}\,E$ (cf.~{\hyperref[lem_del14]{Lemma  \ref*{lem_del14}}}),
$q$ is isomorphic in the generic point of $C_X$ and  $\widetilde{C}\not\subseteq\textup{Supp}\,F$. 
Then we have
\begin{equation}\label{equ_deform}
K_{\widetilde{E}}\cdot \widetilde{C}\leqslant K_E\cdot C_X=(K_X+E)\cdot C_X\leqslant -2,	
\end{equation}
since both $K_X$ and $E$ are Cartier (cf.~{\hyperref[lem_canonical_terminal]{Lemma  \ref*{lem_canonical_terminal}}}). 
By \cite[Chapter 2, Theorem 1.15]{Kollar96},  
$$\dim_{\widetilde{C}}\textup{Chow}(\widetilde{E})\geqslant -K_{\widetilde{E}}\cdot \widetilde{C}-\chi(\mathcal{O}_{\widetilde{C}_N})\geqslant -K_{\widetilde{E}}\cdot \widetilde{C}-1\geqslant 1,$$
where  $\widetilde{C}_N$ is the normalization of $\widetilde{C}$.
Then, $\widetilde{C}$ (and hence its push-forward $\tau_*(\widetilde{C})=C_0$) deforms, which  implies $e\leqslant 0$.
Let us consider the following equalities, noting that $-K_X\cdot \ell=-E\cdot \ell=1$ for a general fibre $\ell\subseteq E$ of $\varphi$.
\begin{align*}
	&q^*(-K_X|_E)=\tau^*(C_0+\alpha f)+\sum a_iP_i,\\
	&q^*(-E|_E)=\tau^*(C_0+\beta f)+\sum b_iP_i,\\
	&q^*(-K_E)=\tau^*(2C_0+(\alpha+\beta)f)+\sum(a_i+b_i)P_i,\\
	&q^*((K_X-E)|_E)=\tau^*((\beta-\alpha)f)+\sum(b_i-a_i)P_i,
\end{align*}
with $P_i$ being $\tau$-exceptional.
On the one hand, the strict nefness of $-K_X|_E$ and $-E|_E$ gives  that $q^*(-K_X|_E)\cdot \tau^*(C_0)>0$ and $q^*(-E|_E)\cdot \tau^*(C_0)>0$; hence $\alpha-e>0$ and $\beta-e>0$.
On the other hand, since $\varphi^*K_{X'}|_E=(K_X-E)|_E$ is nef, we have $\beta-\alpha=q^*((K_X-E)|_E)\cdot \tau^*(C_0)\geqslant  \varphi^*K_{X'}|_E\cdot C_X\geqslant 0$ \hypertarget{1ag}{\textbf{(\dag)}}.

Since $K_{\widetilde{E}}\sim_{\mathbb{Q}}q^*K_E-F$,
there is an inclusion of the canonical sheaf $\omega_{\widetilde{E}}\subseteq q^*(\omega_E)$ as shown in {\hyperref[prop_Q-Goren_surface]{Proposition  \ref*{prop_Q-Goren_surface}}} (cf.~ \cite[Lemma 5-1-9]{KMM87} and \cite[(4.1)]{Sak84}); hence  we have 
$\omega_{\widetilde{E}_m}=\tau_*(\omega_{\widetilde{E}})\subseteq (\tau_*q^*(\omega_E))^{\vee\vee}$. 
Consequently, we get the following inequality 
$$-2C_0-(\alpha+\beta)f\geqslant -2C_0+(2g-2-e)f,$$ 
with $g=g(C_0)$ being the genus.  
So $\alpha+\beta+2g-2-e\leqslant 0$ \hypertarget{2ag}{\textbf{(\dag\dag)}}.
Furthermore,  $C_0+\alpha f$  being  strictly nef on $\widetilde{E}_m$ by the projection formula, 
 we  have $(C_0+\alpha f)^2\geqslant 0$. 
This gives that $\alpha\geqslant e/2$.
Together with \hyperlink{1ag}{\textbf{(\dag)}} and \hyperlink{2ag}{\textbf{(\dag\dag)}}, our $g\leqslant 1$.

It is known that strictly nef divisors on (minimal) ruled surfaces over curves of genus $\leqslant 1$ are indeed ample (cf.\,e.g.\,\cite[Example 1.23 (1)]{KM98}).
Therefore, $(C_0+\alpha f)^2>0$, and thus $\beta\geqslant  \alpha>e/2$ (cf.~\hyperlink{1ag}{\textbf{(\dag)}}).
Return back to \hyperlink{2ag}{\textbf{(\dag\dag)}}, we have $g<1$.
So $g=0$ and $C_0\cong\mathbb{P}^1$.
Since $C_0$ deforms, we have $e=0$, $\widetilde{E}_m\cong\mathbb{P}^1\times\mathbb{P}^1$ and $\alpha+\beta\leqslant 2$ (cf.~\hyperlink{2ag}{\textbf{(\dag\dag)}}). 
Since  $K_X$ is Cartier, both $\alpha$ and $\beta$ are positive integers; hence $\alpha=\beta=1$.
Then we have \hypertarget{1*}{\textbf{(*)}}: $\varphi^*K_{X'}|_E=(K_X-E)|_E\equiv 0$ (cf.~\hyperlink{1ag}{\textbf{(\dag)}}), and thus $\varphi^*L_{X'}|_E\equiv 0$.

\textbf{We claim  that $\tau$ is an isomorphism.} 
Suppose  $\tau^*(C_0)=\widetilde{C}+\sum P_i$ with $P_i$ being $\tau$-exceptional.
Here, our $C_0$ can be chosen as any horizontal section containing some blown-up points of $\tau$ since $\widetilde{E}_m\cong\mathbb{P}^1\times\mathbb{P}^1$.
Then it follows from the projection formula that $\widetilde{C}^2<C_0^2=0$, a contradiction to  {\hyperref[equ_deform]{Equation  (\ref*{equ_deform})}}. 
So $\tau$ is isomorphic as claimed.

Now that $L_{X'}$ is nef, for $t\gg 1$, $K_{X'}+tL_{X'}$ is nef by the projection formula, noting that $\varphi^*(K_{X'}+tL_{X'})=(K_X+tL_X)+(\nu t-1)E$ and $\varphi^*K_{X'}|_E\equiv \varphi^*L_{X'}|_E\equiv 0$ as shown in \hyperlink{1*}{\textbf{(*)}}.
If $K_{X'}+tL_{X'}$ is  big for $t\gg 1$, then 
with the same proof of {\hyperref[lem_L.C>0]{Lemma \ref*{lem_L.C>0}}}, our $K_X+tL_X$ is ample, contradicting our assumption.
So we have $(K_{X'}+tL_{X'})^3=0$ for any $t\gg 1$. 
This in turn implies $K_{X'}^3=0$. 
On the other hand, we note that $K_X^3=0$ (cf.~{\hyperref[lem-big-ample]{Lemma  \ref*{lem-big-ample}}}). 
So we  get a contradiction (cf.~\hyperlink{1*}{\textbf{(*)}} and {\hyperref[lem_canonical_terminal]{Lemma  \ref*{lem_canonical_terminal} (6)}}):
$$0=K_X^3=(\varphi^*K_{X'}+E)^3=E^3=(q^*(E|_E))^2=(C_0+f)^2=2.$$
So our lemma is proved.
\end{proof}

\begin{lemme}\label{lem_intersect_C_<-1}
Suppose that $K_X+tL_X$ is not ample for any $t\gg 1$.
Then, for any curve $B'\subseteq X'$ such that $D_{X'}\cdot B'=0$, we have $K_{X'}\cdot B'\leqslant -1$.
In particular, if $B'\cap C'\neq\emptyset$ (as sets), then $K_{X'}\cdot B'\leqslant -2$.
\end{lemme}

\begin{proof}
By {\hyperref[lem_D.C>0]{Lemma \ref*{lem_D.C>0}}}, $B'\neq C'$.
Denote by $\hat{B'}$ the $\varphi$-proper transform of $B'$ on $X$.	
Then $D_X\cdot \hat{B'}=0$ and thus $K_{X}\cdot \hat{B'}\leqslant -1$, noting that $L_X$ is strctly nef and  $K_X$ is Cartier (cf.~{\hyperref[lem_canonical_terminal]{Lemma  \ref*{lem_canonical_terminal}}}).
Since $E\cdot \hat{B'}\geqslant 0$, we see that $K_{X'}\cdot B'=(K_X-E)\cdot \hat{B'}\leqslant -1$.
In particular, if $B'\cap C'\neq\emptyset$, then $E\cdot \hat{B'}\geqslant 1$, and hence $K_{X'}\cdot \hat{B'}\leqslant -2$. 
\end{proof}

\begin{lemme}\label{lem_not_big_X'}
Either $K_X+tL_X$ is ample for $t\gg 1$, or $D_{X'}$ is not strictly nef.
In the latter case, there is an extremal ray $\mathbb{R}_{\geqslant 0}[\ell']$ on $X'$ such that $K_{X'}\cdot\ell'<0$ and $D_{X'}\cdot \ell'=0$.
\end{lemme}

\begin{proof}
We may assume that $K_X+tL_X$ is not ample for any $t\gg 1$. 
Suppose that $D_{X'}$ is  strictly nef.
By induction (and {\hyperref[remark_conclude_excep_tocurve]{Remark    \ref*{remark_conclude_excep_tocurve}}}), $K_{X'}+t_0D_{X'}$ is ample for any (fixed) $t_0\gg 1$.
But then, $K_X+t_0D_X=\varphi^*(K_{X'}+t_0D_{X'})+E$ is big; thus $K_X+uL_X=\frac{1}{1+t_0\nu}(K_X+t_0D_X)+(u-\frac{t_0}{t_0\nu+1})L_X$ is also big for  $u\gg 1$, a contradiction (cf.~{\hyperref[lem-not-big-some]{Lemma  \ref*{lem-not-big-some}}}). 
Hence, there is an irreducible curve $B'\in\overline{\textup{NE}}(X')$ such that $D_{X'}\cdot B'=0$.
By 
{\hyperref[lem_intersect_C_<-1]{Lemma \ref*{lem_intersect_C_<-1}}},
$K_{X'}\cdot B'<0$.
Since $D_{X'}$ is nef, by the cone theorem (cf.~\cite[Theorem 3.7]{KM98}), there exists a $K_{X'}$-negative extremal curve $\ell'$ such that $K_{X'}\cdot\ell'<0$ and $D_{X'}\cdot\ell'=0$. 
\end{proof}

\begin{lemme}\label{lem-sec-con-surface}
If $K_X+tL_X$ is not ample for any $t\gg 1$, then the contraction $\varphi':X'\to X''$ of $\mathbb{R}_{\geqslant  0}[\ell']$ in {\hyperref[lem_not_big_X']{Lemma \ref*{lem_not_big_X'}}} is birational.	
\end{lemme}

\begin{proof}
Suppose the contrary that $\dim X''\leqslant 2$.
By the cone theorem (cf.~\cite[Theorem 3.7]{KM98}), $D_{X'}=\varphi'^*(D_{X''})$ for some nef divisor $D_{X''}$ on $X''$.
If $X''$ is a point, then $\rho(X')=1$ and $D_{X'}$ is ample, a contradiction to {\hyperref[prop-q-effective]{Proposition  \ref*{prop-q-effective}}}.
If $X''$ is a curve or  $\dim X''=2$ and $D_{X''}^2\neq 0$ (and hence big),
then $D_{X'}$ (and hence $D_X$) is numerically equivalent to a non-zero effective divisor, contradicting  {\hyperref[prop-q-effective]{Proposition  \ref*{prop-q-effective}}} again.

Suppose that $\dim X''=2$ and $D_{X''}^2=0$. 
We claim that the composite $\varphi'\circ\varphi:X\to X''$ is an equi-dimensional $K_X$-negative contraction (of an extremal face); thus we get a contradiction to our assumption by {\hyperref[thm_contr_surface_conic]{Theorem    \ref*{thm_contr_surface_conic}}}. 
First, by the cone theorem (cf.~\cite[Theorem 3.7]{KM98}), $\varphi$ is equi-dimensional (of relative dimension one).
Since $C'$ is not contracted by $\varphi'$ (cf.~{\hyperref[lem_D.C>0]{Lemma \ref*{lem_D.C>0}}}), the composite $\varphi'\circ\varphi$ is also equi-dimensional. 
Take any irreducible curve $F$ contracted by $\varphi'\circ\varphi$.
Then $D_X\cdot F=(\varphi'\circ\varphi)^*(D_{X''})\cdot F=0$ and thus $K_X\cdot F<0$.
Hence, $-K_X$ is $(\varphi'\circ\varphi)$-ample  (cf.~\cite[Theorem 1.42]{KM98}) and 
$\varphi'\circ\varphi$ is a $K_X$-negative contraction.
Finally, the contraction $X\to X''$ is  clearly extremal by 
considering the pullback of any ample divisor on $X''$. 
\end{proof}

The following lemma is a bit technical. 
We divide the proof into several cases for readers' convenience. 
We shall heavily apply the terminalization and {\hyperref[lem_canonical_terminal]{Lemma  \ref*{lem_canonical_terminal}}}.
Recall that the \textit{length} of a $K_X$-negative extremal contraction is defined to be the minimum of $-K_X\cdot B$ for generic curves $B$ in the covering families of contracted locus.

\begin{lemme}\label{lem_bir-intersectC}
If $K_X+tL_X$ is not ample for any $t\gg 1$, then the contraction $\varphi':X'\to X''$ of $\mathbb{R}_{\geqslant  0}[\ell']$ in {\hyperref[lem_not_big_X']{Lemma \ref*{lem_not_big_X'}}} is  divisorial  with the exceptional divisor  $E'$ such that   $E'\cap C'=\emptyset$.
\end{lemme}

\begin{proof}
By {\hyperref[lem-sec-con-surface]{Lemma  \ref*{lem-sec-con-surface}}}  and {\hyperref[lem_del14]{Lemma  \ref*{lem_del14}}},  our $\varphi'$ is  a divisorial contraction.
In the following, we shall discuss case-by-case in terms of $E'$ and the intersection of $E'\cap C'$.
\par \vskip 0.4pc \noindent
\textbf{Case (1). Suppose $C'\subseteq E'$.}  
Then $C'$ being rigid and $D_{X'}\cdot C'>0$ (cf.~{\hyperref[lem_L.C>0]{Lemma \ref*{lem_L.C>0}}} and {\hyperref[lem_D.C>0]{Lemma \ref*{lem_D.C>0}}}) would imply that $\varphi'$ is a blow-up of a curve on $X''$ (cf.~{\hyperref[lem_canonical_terminal]{Lemma  \ref*{lem_canonical_terminal}}}) with $C'$ being  horizontal  on $E'\subseteq X'$. 
Let $\ell'$ be a general fibre of $\varphi'$. 
Since $D_{X'}\cdot \ell'=0$ and $\ell' \cap C'\neq\emptyset$, we have $K_{X'}\cdot\ell'\leqslant -2$ (cf.~{\hyperref[lem_intersect_C_<-1]{Lemma \ref*{lem_intersect_C_<-1}}}),  contradicting {\hyperref[lem_del14]{Lemma  \ref*{lem_del14} (3)}}.

\par \vskip 0.4pc \noindent
\textbf{Case (2). Suppose  $E'\cap C'$ is a finite non-empty  set.} 
Then by {\hyperref[lem_canonical_terminal]{Lemma  \ref*{lem_canonical_terminal}}}, we have the following commutative diagram
\begin{align}\label{diagram_4}\tag{$*$}
\xymatrix{Y'\ar[r]^{\phi'}\ar[d]_{\tau'}&Y''\ar[d]^{\tau''}\\
X'\ar[r]_{\varphi'}&X''
}	
\end{align}
where $\tau'$ is the crepant terminalization, $\phi'$ is a divisorial contraction, and $Y'$ has only $\mathbb{Q}$-factorial Gorenstein terminal singularities.

\par \vskip 0.4pc \noindent
\textbf{(2i).} Suppose $\phi'$ maps the exceptional divisor $E_{Y'}$ to a point  and the length of  $\phi'$ is one (this is the case when \cite[Theorem 5 (2), (3) or (4)]{Cut88} happen). 
By {\hyperref[lem_canonical_terminal]{Lemma  \ref*{lem_canonical_terminal}}}, $E'\cong E_{Y'}$.  
Then we can pick a  
curve $\ell'\subseteq E'\subseteq X'$ of $\varphi'$ meeting $C'$ such that $K_{X'}\cdot\ell'=-1$ (noting that $\tau'$ is crepant) and $D_{X'}\cdot \ell'=0$. 
However, this contradicts 
{\hyperref[lem_intersect_C_<-1]{Lemma \ref*{lem_intersect_C_<-1}}}.

\par \vskip 0.4pc \noindent 
\textbf{(2ii).} Suppose  $\phi'$ maps the exceptional divisor $E_{Y'}$ onto a curve   (and hence the length of  $\phi'$ is still one, which is the case when \cite[Theorem 4]{Cut88} happens).
Let $P\in E'\cap C'$.
If $P$ is a terminal point, then we 
pick a fibre $\ell_{Y'}\subseteq E_{Y'}$ passing through $\tau'^{-1}(P)$.
If $P$ is not a terminal point, then we take a curve $c_0\subseteq \tau'^{-1}(P)\cap E_{Y'}$, which is a horizontal curve of $\phi'$; in this case, we pick $\ell_{Y'}$ to be any fibre of $\phi'$, which automatically intersects with $c_0$.
In both cases, let $\ell':=\tau'(\ell_{Y'})\ni P$.  
Similarly, $K_{X'}\cdot\ell'=\tau'^*K_{X'}\cdot \ell_{Y'}=K_{Y'}\cdot \ell_{Y'}=-1$, which contradicts 
{\hyperref[lem_intersect_C_<-1]{Lemma \ref*{lem_intersect_C_<-1}}} again.

\par \vskip 0.4pc \noindent
\textbf{(2iii).} By the classification, we may assume that $\phi'$ maps the exceptional divisor $E_{Y'}$ to a point and the length of $\phi'$ is two (this is the case when \cite[Theorem 5 (1)]{Cut88} happens). 
In this case, $E_Y'\cong\mathbb{P}^2$ with $E_{Y'}|_{E_{Y'}}=\mathcal{O}(-1)$, $Y''$ (and hence $Y'$) is  smooth, and $\varphi'(E')$ is a point. 
Consider the following  (by taking $\tau_2:=\tau'$ and $\tau_3:=\tau''$ in the diagram (\ref{diagram_4})):
\begin{align}\label{diagram_5}\tag{$**$}
\xymatrix{Z\ar[d]_{\tau_0}&Y\ar[r]^{\phi}\ar[l]_{\psi_Y}\ar[d]_{\tau_1}&Y'\ar[r]^{\phi'}\ar[d]^{\tau_2}&Y''\ar[d]^{\tau_3}\\
W&X\ar[l]^{\psi_X}\ar[r]_\varphi &X'\ar[r]_{\varphi'}&X''
}	
\end{align}

\textbf{We first explain in the following how to get the  diagram (\ref{diagram_5}).}
By {\hyperref[lem_canonical_terminal]{Lemma  \ref*{lem_canonical_terminal}}}, we get the diagrams $\tau_2\circ\phi=\varphi\circ\tau_1$ and $\tau_3\circ\phi'=\varphi'\circ\tau_2$,
noting that $\textup{Exc}(\tau_1)$ is disjoint with $E_Y:=\textup{Exc}(\phi)$ and $\textup{Exc}(\tau_2)$ is  disjoint with $E_{Y'}:=\textup{Exc}(\phi')$.
This implies that $E'\cong E_{Y'}=\mathbb{P}^2$ and  $X'$ is smooth around $E'$.
Fixing a line $\ell'\subseteq E'$ of $\varphi'$ meeting $C'$ and taking $\hat{\ell'}$ to be its proper transform on  $X$,  we have
$K_{X'}\cdot \ell'=-2$.
Let $a:=\hat{\ell'}\cdot E\geqslant 1$. 
Then $K_X\cdot\hat{\ell'}=-2+a$ and $L_X\cdot\hat{\ell'}=\nu(2-a)>0$.
In particular, $a=1$ and $K_X\cdot\hat{\ell'}=-1$.
\textbf{This also implies that $E'$  meets $C'$ along a  single (smooth) point of $C'$.}
Indeed, if $\sharp (E'\cap C')\geqslant 2$, then the line $\ell'\subseteq E'$ passing through any two points of $E'\cap C'$ satisfies $E\cdot\hat{\ell'}\geqslant 2$, which is absurd. 
So the $\varphi$-proper transform $\hat{E'}\cong\mathbb{F}_1$ of $E'$ is ruled over $h_W\cong\mathbb{P}^1$ with fibres $\hat{\ell'}$ and the negative section $C_0$ (being a fibre of $\varphi$). 
Suppose $P\in E'\cap C'$ is a singular point  of $C'$.  
Then   $E'\cdot C'=\deg \mathcal{O}_{X'}(E')|_{\widetilde{C'}}\geqslant 2$, with $\widetilde{C'}$  the normalization of $C'$, in which case,
any line $\ell'\subseteq E'$ containing $P$ satisfies  $E\cdot\hat{\ell'}=(E|_{\hat{E'}}\cdot\hat{\ell'})_{\hat{E'}}=(C'\cdot E')\geqslant 2$, a contradiction.
So  $E'$ meets $C'$ along a single smooth point of $C'$.

Since $\hat{E'}\cdot\hat{\ell'}=\varphi^*E'\cdot\hat{\ell'}=-1$ (noting that $E'|_{E'}=\mathcal{O}(-1)$) and  $-\hat{E'}|_{\hat{E'}}$ is  relative ample with respect to the ruling $\hat{E'}\to h_W$, we obtain a contraction $\psi_X:X\to W$   such that $\psi_X|_{\hat{E'}}$ coincides with this ruling and $\psi_X|_{X\backslash\hat{E'}}\cong W\backslash h_W$ (cf.~e.g.~\cite[Proposition 7.4]{HP16}). 
Similar to $\psi_X$, the morphism $\psi_Y$ is induced by contracting the divisor $\hat{E}_{Y'}:=\tau_1^*\hat{E'}\cong\mathbb{F}_1$ to $h_Z\cong h_W\cong\mathbb{P}^1$. 
By the rigidity lemma, $\psi_X\circ\tau_1$ factors through $\psi_Y$ 
and we get $\tau_0$.

\textbf{Caution: it is still not clear whether $W$ is projective or not, and  $\psi_X$ here may not be an extremal contraction.
Therefore, we could not apply the induction on $W$ so far.} 

Since $X$ has only canonical singularities and $L_X+\nu\hat{E'}$ is $\psi_X$-trivial,  
by \cite[Theorem 4.12]{Nak87}, there exists a  divisor  $L_W$ on $W$ such that
$L_X=\psi_X^*(L_W)-\nu\hat{E'}$. 
Recall that $C_0$ (a fibre of $\varphi$) is the negative section of $\hat{E}'$. 
Since $-{\hat{E'}}|_{{\hat{E'}}}=C_0+\hat{\ell'}$ (noting that $\hat{E'}\cdot\hat{\ell'}=E'\cdot\ell'=-1$ and $\hat{E'}\cdot C_0=0$), we obtain the following strictly nef divisor on $\hat{E'}$
$$L_X|_{\hat{E}'}=\nu C_0+(L_W\cdot h_W+\nu)\hat{\ell'}.$$
Since  a strictly nef divisor on $\mathbb{F}_1$ is ample, our $L_X|_{\hat{E}'}$ is ample and hence  $L_W\cdot h_W>0$ (cf.~\cite[Chapter 5, Proposition 2.20]{Har77}). 
Then $L_W$ is strictly nef on the Moishezon threefold $W$ by the projection formula. 
We consider the following commutative diagram:
\[\xymatrix{
Z_0\ar[d]_{\sigma_0}\ar@/_2pc/[dd]_{\pi_0}&Y_0\ar[l]_{\psi_{Y_0}}\ar[d]^{\sigma_1}\ar@/^2pc/[dd]^{\pi_1}\\
Z\ar[d]_{\tau_0}&Y\ar[d]^{\tau_1}\ar[l]_{\psi_Y}\\
W&X\ar[l]^{\psi_X}
}
\]
where $\sigma_1:Y_0\to Y$ is a resolution with $Y_0$ being smooth.
Since $X$ is smooth around $\hat{E'}$ (recalling that $E'\cap C'$  is a smooth point of $C'\subseteq X'$), our $Y$ is also smooth around $\hat{E}_{Y'}$.
So  $\sigma_1$ is isomorphic around $\hat{E}_{Y'}$.
In particular, $Y_0$ admits a contraction $\psi_{Y_0}$ mapping $E_{Y_0}:=\sigma_1^*(\hat{E}_{Y'})=\pi_1^*\hat{E'}\cong\mathbb{F}_1$ onto a curve $h_{Z_0}\cong\mathbb{P}^1$ (cf.~e.g. \cite[Proposition 7.4]{HP16}).
By the rigidity lemma, $\psi_Y\circ\sigma_1$ factors through $\psi_{Y_0}$, and we get the induced $\sigma_0$.
Since  $Y_0$ is smooth and the conormal sheaf of $E_{Y_0}$ is isomorphic to $-E_{Y_0}|_{E_{Y_0}}$ which is (locally over $h_{Z_0}$) isomorphic to $\mathcal{O}(1)$, 
it follows from \cite[Corollary 6.11]{Art70} that $Z_0$ is smooth.
\begin{claim}\label{claim_proj}
$Z_0$ is projective.	
\end{claim}

Suppose the claim for the time being. 
Since $-K_X$ is $\psi_X$-ample, we have $R^j(\psi_X)_*\mathcal{O}_X=0$ for all $j\geqslant 1$ (cf.~\cite[Theorem 1-2-5]{KMM87}) and  hence $W$ has only isolated rational singularities.
Since $Z_0\to W$ is a resolution and $Z_0$ is projective by the assumption, 
by \cite[Remark 3.5]{HP16}, $W$ admits a (smooth) K\"ahler form and hence $W$ is both K\"ahler and Moishezon.
So the projectivity of $W$ follows from \cite[Theorem 6]{Nam02}.
Now, $W$ being projective and $L_W$ being strictly nef, we get a contradiction by (Proof of) {\hyperref[lem_L.C>0]{Lemma \ref*{lem_L.C>0}}}.

\par \vskip 1pc \noindent
\textbf{Proof of {\hyperref[claim_proj]{Claim \ref*{claim_proj}}}  (End of Proof of {\hyperref[lem_bir-intersectC]{Lemma \ref*{lem_bir-intersectC}}}).}
Suppose the contrary. 
Denote by $\pi_i:=\tau_i\circ\sigma_i$.
By  \cite[Theorem 2.5]{Pet86}, there is an irreducible curve $b$ and a positive closed current $T$ on $Z_0$ such that $b+T\equiv0$ (as $(2,2)$-currents); hence $(\pi_0)_*(b+T)\equiv0$.
Since $L_W$ is strictly nef on $W$, our $b$ is $\pi_0$-exceptional.
Note that $X$ is smooth around $\hat{E'}$ and $\hat{E}_{Y'}\cap \textup{Exc}(\tau_1)=\emptyset$.
So $E_{Y_0}\cap\textup{Exc}(\pi_1)=\emptyset$.
Take a very ample divisor $H$ on $Y_0$.
Since $Z_0$ is $\mathbb{Q}$-factorial, it is easy to verify that $(\psi_{Y_0})_*H\cdot b>0$ (cf.~\cite[Lemma 2.62]{KM98}), noting that $\pi_0^*L_W\cdot h_{Z_0}>0$   and thus $b\neq h_{Z_0}:=\psi_{Y_0}(E_{Y_0})$.
So  $(\psi_{Y_0})_*H\cdot T<0$.
By \cite{Siu74}, our $T=\chi_{h_{Z_0}}T+\chi_{Z_0\backslash h_{Z_0}}T=\delta h_{Z_0}+\chi_{Z_0\backslash h_{Z_0}}T$, where $\chi_{h_{Z_0}}T$ and $\chi_{Z_0\backslash h_{Z_0}}T$ are positive closed currents. 
Since $(\psi_{Y_0})_*H\cdot \chi_{Z_0\backslash h_{Z_0}}T\geqslant 0$, we have $\delta>0$.
Then
$$0\equiv (\pi_0)_*(b+T)\equiv (\pi_0)_*T=\delta h_W+(\pi_0)_*(\chi_{Z_0\backslash h_{Z_0}}T).$$
 Since $L_W\cdot h_W>0$ and $L_W\cdot(\pi_0)_*(\chi_{Z_0\backslash h_{Z_0}}T)\geqslant 0$, the above equality is absurd.
\end{proof}

\begin{proof}[\textup{\textbf{End of Proof of {\hyperref[thm_3fold_surface_curve]{Theorem \ref*{thm_3fold_surface_curve}}}.}}]
We suppose the contrary that $K_X+tL_X$ is not ample for any $t\gg 1$. 
By {\hyperref[lem-sec-con-surface]{Lemma \ref*{lem-sec-con-surface}}}
 and {\hyperref[lem_bir-intersectC]{Lemma \ref*{lem_bir-intersectC}}}, 
 $\varphi':X'\to X''$ is a divisorial contraction with the exceptional divisor $E'$  disjoint with $C'$.
Then the strict transform of $E'$ on $X$ is some $E_j$, $j\in I_0$ (cf.~{\hyperref[not_bir_surface_curve]{Notation  \ref*{not_bir_surface_curve}}}); hence we can continue to consider $D_{X''}$.
By {\hyperref[prop-q-effective]{Proposition  \ref*{prop-q-effective}}} and the induction  on $X''$ (cf.~{\hyperref[remark_conclude_excep_tocurve]{Remark    \ref*{remark_conclude_excep_tocurve}}}),  our $D_{X''}$ is  not strictly nef. 
So we get the third  contraction $X''\to X'''$ (cf.~{\hyperref[lem_not_big_X']{Lemma \ref*{lem_not_big_X'}}}). 
If $\dim X'''\leqslant 2$, then we  argue as in {\hyperref[lem-sec-con-surface]{Lemma \ref*{lem-sec-con-surface}}}, together with {\hyperref[thm_contr_surface_conic]{Theorem    \ref*{thm_contr_surface_conic}}}, to conclude that $K_X+tL_X$ is ample, a contradiction to our assumption.   
So $X''\to X'''$ is still birational with the exceptional divisor $E''$. 
But then, {\hyperref[lem_bir-intersectC]{Lemma \ref*{lem_bir-intersectC}}} Case (1) shows that neither $\varphi'(C')$ (where $C'=\varphi(E)\subseteq X'$) nor $C''$ (where $C''=\varphi'(E')\subseteq X''$) is contained in $E''$.
Further,  {\hyperref[lem_bir-intersectC]{Lemma \ref*{lem_bir-intersectC}}} Case (2) gives that
 such $E''$ cannot intersect with  $\varphi'(C')\cup C''$. 
Hence,  $E$, $E'$ and $E''$ are pairwise disjoint. 
Since  $X$ is uniruled, after finitely many steps, we get some $X_M$ with $\dim X_M\leqslant 2$. 
By {\hyperref[thm_contr_surface_conic]{Theorem    \ref*{thm_contr_surface_conic}}}, $K_X+tL_X$ is ample which contradicts our assumption (cf.~{\hyperref[lem-sec-con-surface]{Lemma \ref*{lem-sec-con-surface}}}).
\end{proof}

\begin{proof}[Proof of {\hyperref[main_theorem_Goren_ter_3fold]{Theorem   \ref*{main_theorem_Goren_ter_3fold} (4)}}]
It follows  from {\hyperref[remark_conclude_excep_tocurve]{Remark    \ref*{remark_conclude_excep_tocurve}}} and 	{\hyperref[thm_3fold_surface_curve]{Theorem \ref*{thm_3fold_surface_curve}}}.
\end{proof}

\begingroup
\setstretch{1.0}

\bibliographystyle{alpha}
\bibliography{strict-nef}

\endgroup
\end{document}